\documentclass{article}
\usepackage[T1]{fontenc}
\usepackage{soul}
\usepackage{amsmath}
\usepackage{lmodern}
\usepackage{amsthm}
\usepackage{authblk}

\usepackage{bbding}
\usepackage[ruled,vlined,linesnumbered]{algorithm2e}
\DontPrintSemicolon
\SetKwInput{Input}{Input}
\SetKwInput{Output}{Output}
\SetKw{To}{ to }
\SetKw{Print}{Print }
\SetAlFnt{\small}

\usepackage{amssymb}

\usepackage{algpseudocode}
\usepackage{graphics}
\usepackage{bbm}
\usepackage{tikz}
\usepackage{bm}
\usepackage{appendix}
\usepackage{hyperref}

\usepackage{graphicx}
\usepackage{subfig}

\usepackage{multirow}
\usepackage{rotating}
\usepackage{verbatim}
\usepackage{optidef}
\usepackage{cleveref}
\usepackage{array}
\usepackage{tabularx}
\usepackage{arydshln}
\usepackage{easybmat}
\usepackage{enumitem}
\usepackage{xcolor}
\usetikzlibrary{positioning}
\newcommand*\xbar[1]{%
  \hbox{%
    \vbox{%
      \hrule height 0.5pt 
      \kern0.5ex
      \hbox{%
        \kern-0.1em
        \ensuremath{#1}%
        \kern-0.1em
      }%
    }%
  }%
} 

\usepackage[T1]{fontenc}
\usepackage{graphicx}
\newcommand{\R}{\mathbb{R}}
\newcommand{\conv}[0]{{\mathbf{conv}}}
\newcommand{\proj}[0]{{\mathbf{proj}}}

\newcommand{\OPT}{\textup{OPT}}
\newcommand{\LAG}{\textup{DUAL}}

\newcommand{\cT}{\mathcal{T}}

\newcommand{\uv}[0]{\mathbf{u}}
\newcommand{\vv}[0]{\mathbf{v}}
\newcommand{\xv}[0]{\mathbf{x}}
\newcommand{\yv}[0]{\mathbf{y}}
\newcommand{\pv}[0]{\mathbf{p}}

\newcommand{\iprod}[2]{\left\langle {#1}, {#2} \right\rangle}

\DeclareMathOperator{\argmin}{\operatorname{argmin}}

\newcommand{\CS}{\mathcal{S}}
\newcommand{\CQ}{\mathcal{Q}}
\newcommand{\CX}{\mathcal{X}}

\newcommand{\CT}{\mathcal{T}}

\newcommand{\norm}[1]{\left\lVert#1\right\rVert}

\newcommand{\dv}[0]{\mathbf{d}}
\newcommand{\ev}[0]{\mathbf{e}}

\newcommand{\cv}[0]{\mathbf{c}}

\newcommand{\qv}[0]{\mathbf{q}}
\newcommand{\bv}[0]{\mathbf{b}}
\newcommand{\gv}[0]{\mathbf{g}}

\newcommand{\wv}[0]{\mathbf{w}}

\newcommand{\hv}[0]{\mathbf{h}}

\usepackage{booktabs}
\usepackage{pgf}
\pgfmathsetmacro{\backspace}{-0.5}

\let\subparagraph\paragraph
\usepackage[compact]{titlesec}
    \titlespacing{\section}{0pt}{2ex}{1ex}
    \titlespacing{\subsection}{0pt}{1ex}{0ex}
    \titlespacing{\subsubsection}{0pt}{0.5ex}{0ex}

\newcolumntype{L}{>{\centering\arraybackslash}m{4cm}}
\newcolumntype{R}{>{\centering\arraybackslash}m{2cm}}

\newtheorem{theorem}{Theorem}
\newtheorem{proposition}[theorem]{Proposition}%
\newtheorem{remark}{Remark}%
\newtheorem{assumption}{Assumption}%
\newtheorem{corollary}{Corollary}%

\newtheorem{definition}{Definition}%
\newtheorem{lemma}{Lemma}

\providecommand{\keywords}[1]
{
  \small	
  \textbf{\textbf{Keywords: }} #1
}

\title{Lagrangian dual with zero duality gap that admits decomposition}
\author[1]{Diego Cifuentes}
\author[1]{Santanu S. Dey}
\author[1]{Jingye Xu}
\affil[1]{H. Milton Stewart School of Industrial and Systems Engineering, Georgia Institute of Technology, GA, USA}
\date{}

\begin{document}

\maketitle
\abstract{ For mixed integer programs (MIPs) with block structures and coupling constraints, on dualizing the coupling constraints the resulting Lagrangian relaxation becomes decomposable into blocks which allows for the use of parallel computing. However, the resulting Lagrangian dual can 
    have non-zero duality gap due to the inherent non-convexity of MIPs. In this paper, we propose two reformulations of such MIPs by adding redundant constraints, such that the Lagrangian dual obtained by dualizing the coupling constraints and the redundant constraints have zero duality gap while still remaining decomposable. One of these reformulations is similar, although not the same as the RLT hierarchy. In this case, we present multiplicative bounds on the quality of the dual bound at each level of the hierarchy for packing and covering MIPs. We show our results are applicable to general sparse MIPs, where decomposability is revealed via the tree-decomposition of the intersection graph of the constraint matrix. In preliminary experiments, we observe that the proposed Lagrangian duals give better dual bounds than classical Lagrangian dual and Gurobi in equal time, where Gurobi is not exploiting decomposability.}



\section{Introduction}
Consider a two-block mixed integer program (MIP) with coupling constraints of the following form:
\begin{subequations}
\label{prob}
\begin{align}
\label{eq:obj}\OPT :=  \min_{(\xv,\yv)} \ & \sum_{i \in \{1,2\}} \iprod{\cv^{(i)}}{\xv^{(i)}} + \iprod{\dv^{(i)}}{\yv^{(i)}} \\
\label{eq:sub}\text{s.t.}\ & (\xv^{(i)},\yv^{(i)}) \in \mathcal{X}^{(i)},\forall i \in \{1,2\}, \\
\label{prob_link_constr}
\ & \xv^{(1)} = \xv^{(2)} \in \{0,1\}^n.
\end{align}
\end{subequations}
where $\mathcal{X}^{(i)} := \left\{ (\xv^{(i)},\yv^{(i)}) \;\middle\vert\;
   \begin{array}{@{}l@{}} A^{(i)} \xv^{(i)} + B^{(i)} \yv^{(i)} \leq \bv^{(i)}, \\
   \yv^{(i)}  \text{ is nonnegative and mixed-integer},\\
   \xv^{(i)} \in \{0,1\}^n
   \end{array} 
\right\}$ with  $A^{(i)}$, $B^{(i)}$, $\bv^{(i)}$, $\cv^{(i)}$, $\dv^{(i)}$ being rational data of suitable dimension for each $i \in \{1,2\}$.

If the coupling constraints (\ref{prob_link_constr}) are ignored, then the remaining problem can be decomposed into independent optimization tasks over each $\mathcal{X}^{(i)
}$. One classic approach that exploits this structure to obtain dual bounds for (\ref{prob}) is that of 
Lagrangian relaxation~\cite{geoffrion2009lagrangean}. Specifically,  by dualizing (\ref{prob_link_constr}), we obtain:
\begin{equation}
    \label{dual_prob_inner}
    \begin{aligned}
        L(\lambda) := \min_{(\xv,\yv)} & \left(\sum_{i \in \{1,2\}} \iprod{\cv^{(i)}}{\xv^{(i)}} + \iprod{\dv^{(i)}}{\yv^{(i)}}\right) + \iprod{\lambda}{\xv^{(1)} - \xv^{(2)}}  \\
        \ & \text{s.t.}\ (\xv^{(i)},\yv^{(i)}) \in \mathcal{X}^{(i)},\forall i \in \{1,2\},
    \end{aligned}
\end{equation}
and 
\begin{equation}
    \label{dual_prob}
    \begin{aligned}
        \LAG:= \max_{\lambda}  L(\lambda).
    \end{aligned}
\end{equation}
It is well-known that $L(\lambda)$ is a non-smooth concave function and sub-gradients for $L(\lambda)$ can be obtained by solving (\ref{dual_prob_inner}), which is a collection of independent optimization tasks over $\mathcal{X}^{(i)}$ \emph{that could be solved in parallel}. Using these sub-gradients, 
one can solve (\ref{dual_prob}) via 
non-smooth optimization methods~\cite{bagirov2014introduction}. 

Even though weak duality always holds, due to non-convexity, strong duality generally fails, that is, $\OPT > \LAG $. 
On the other hand, one can solve (\ref{prob}) directly without exploiting decomposability,  trivially obtaining zero duality gap. 
This motivates the main question of this paper: \emph{Can we obtain zero duality gap using Lagrangian-style relaxations while simultaneously allowing decomposition into sub-problems?}
Our main contributions are:

(1.) Obtaining zero duality gap and decomposability simultaneously: We design two reformulations of (\ref{prob}) whose Lagrangian duals, namely M-Lagrangian dual (Section~\ref{sec:MV}, Section~\ref{sec:M2}) and V-Lagrangian dual (Section~\ref{sec:MV},  Section~\ref{sec:v-converge}), achieve the twin goal of zero duality gap and decomposability. The M-Lagrangian method is a hierarchy of reformulations of (\ref{prob}) similar to the Reformulation-Linearization-Technique (RLT)~\cite{adams1990linearization} but not the same, whose
Lagrangian duals achieve zero duality gap in the last step of the hierarchy, while simultaneously each level admits decomposition into sub-problems.

(2.) Analysis of bounds: We present multiplicative bounds on the duality gap at different levels of the M-Lagrangian hierarchy for packing and covering problems; see Section~\ref{sec:M2PC}.

(3.) Run-time for solving V-Lagrangian dual: We present bounds on the number of iterations needed to solve the V-Lagrangian dual via sub-gradient methods in comparison to an algorithm in~\cite{ahmed2013scenario}, see Section~\ref{sec:v-converge}.

(4.) Generalization to arbitrary MIPs: Consider a loosely coupled general MIP where the block structure is revealed using a tree-decomposition~\cite{diestel2024graph}  of the intersection graph~\cite{fulkerson1965incidence} of the constraint matrix. Here, the blocks correspond to smaller problems defined over variables in the bags of the tree-decomposition. We show how to generalize the above results: simultaneously achieving decomposability and strong duality, and multiplicative bounds for packing and covering instances, for the M-Lagrangian dual to this setting, see Section~\ref{sec:general}.

(5.) Preliminary computational results: We illustrate how the proposed Lagrangian duals can outperform classical Lagrangian relaxation and Gurobi in terms of dual bounds achieved,  where Gurobi solves the whole problem without exploiting decomposition; see Section~\ref{sec:compute}.

Finally, in Section~\ref{sec:conclude} we make concluding remarks.

One another exact algorithm that admits decomposition for (\ref{prob}) (and more generally via tree-decomposition of intersection graph) is the Dynamic Programming (DP) algorithm~\cite{aris1964optimization,bertele1973non,nerohauser1966introduction,bienstock2018lp,laurent2009sums}. However,  DP suffers from the disadvantage that we do not obtain any information on bounds until the algorithm terminates. On the other hand, the advantage of Lagrangian duals is that even if one terminates without reaching optimality, it is possible to retrieve dual bounds. 

\emph{Literature Survey.} There are very few exact algorithms that also achieve decomposability. Apart from DP discussed above, in the special case of two-stage stochastic programming, there are specialized algorithms ~\cite{ahmed2013scenario,zhang2014finitely} and the integer L-shaped method~\cite{laporte1993integer,angulo2016improving}. Attempts to achieve the above two goals simultaneously fail for Augmented Lagrangian method (AL)~\cite{feizollahi2017exact,gu2020exact} which achieves strong duality, but not decomposability as the penalty terms have to be norms. Recently, the paper ~\cite{sun2024decomposition}, develops a variant of AL method that achieves decomposability.  

\emph{Notation.}
For a positive integer $n$, we use $[n]$ to denote the set $\{1, \dots, n\}$. Given a finite set $S$ we represent its power set as $2^{S}$. We denote the collection of subsets of $[n]$ containing all elements of cardinality $k$ or less as
$\binom{[n]}{\leq k}$. Given a graph $G$, we refer to its vertices by $V(G)$ and its edges by $E(G)$.

\section{Decomposable reformulation and its Lagrangian dual for two-block problem.}\label{sec:2block}

\subsection{Two reformulations.}\label{sec:MV}
The main idea of the reformulations is to add redundant constraints to (\ref{prob}) that are separable after dualizing. 
Consider a separable constraint implied by $\xv^{(1)} = \xv^{(2)} \in\{0,1\}^n$, that is,
\begin{equation}
    \begin{aligned}
        \label{implied}
        \xv^{(1)} = \xv^{(2)} \in \{0,1\}^n \implies F^{(1)}(\xv^{(1)}) + F^{(2)}(\xv^{(2)}) \geq 0.
    \end{aligned}
\end{equation}

We extend $\mathcal{X}^{(i)}$ to $\mathcal{X}^{(i)}_{ex}$ in the following fashion: $
        \mathcal{X}^{(i)}_{ex} := \{(\xv^{(i)},\yv^{(i)},w^{(i)}) : (\xv^{(i)},\yv^{(i)}) \in \mathcal{X}^{(i)}, w^{(i)} = F^{(i)}(\xv^{(i)})\},\forall i \in \{1,2\}$,
%
%
and obtain the following reformulation of (\ref{prob}):
\begin{subequations}
\label{ex_main_prob}
\begin{align}
\OPT :=  \min_{(\xv,\yv,\wv)} \ & \sum_{i \in \{1,2\}}  \iprod{\cv^{(i)}}{\xv^{(i)}} + \iprod{\dv^{(i)}}{\yv^{(i)}} \\
\text{s.t.}\ & (\xv^{(i)},\yv^{(i)},w^{(i)}) \in \mathcal{X}^{(i)}_{ex},\forall i \in \{1,2\}, \\
\label{ex_main_prob_link}
\ & \xv^{(1)} = \xv^{(2)}, \\
\label{ex_main_prob_link_extra}
\ & w^{(1)} + w^{(2)} \geq 0.
\end{align}
\end{subequations}

Since the constraint (\ref{ex_main_prob_link_extra}) is also a coupling constraint, we introduce additional multipliers and dualize it as follows:  
\begin{equation}
    \label{ex_dual_prob_inner}
    \begin{aligned}
        L^{ex}(\lambda,\mu) := \min_{(\xv,\yv,\wv)} & \left(\sum_{i \in \{1,2\}}  \iprod{\cv^{(i)}}{\xv^{(i)}} + \iprod{\dv^{(i)}}{\yv^{(i)}}\right) \\ & +   \iprod{\lambda}{\xv^{(1)} - \xv^{(2)}}  + \iprod{\mu}{w^{(1)} + w^{(2)}}\\
        \ & \text{s.t.}\ (\xv^{(i)},\yv^{(i)},w^{(i)}) \in \mathcal{X}^{(i)}_{ex},\forall i \in \{1,2\},
    \end{aligned}
\end{equation}
and
\begin{equation}
    \label{ex_dual_prob}
    \begin{aligned}
        \LAG^{ex} := \max_{\mu \geq 0,\lambda}  L^{ex}(\lambda,\mu).
    \end{aligned}
\end{equation}

Although the addition of (\ref{implied})  has no effect on (\ref{prob}), that is on the primal side the optimal value trivially remains the same; on the dual side they can indeed improve the bound, that is it is possible that $\LAG \neq \LAG^{ex}$.

We first establish a result to identify when $\LAG = \LAG^{ex}.$
The following result is a strengthening on a result from~\cite{lemarechal2001geometric} and Details of the proof are presented in Appendix \ref{sec:redundant_constraint}.
\begin{proposition}   \label{prop:redundant_constraint}
    If either $F^{(1)}(\cdot)$ or $F^{(2)}(\cdot)$ is affine, then $\LAG = \LAG^{ex}.$    
\end{proposition}

Proposition~\ref{prop:redundant_constraint} motivates us to consider the following two different choices of redundant non-affine constraints: 
\begin{enumerate}
    \item \textit{Monomial-based  reformulation}: Let $\mathcal{S} \subseteq 2^{[n]}$ be a collection of subsets of $[n]$, then the primal-redundant constraints take the form of
    \begin{equation}
        \begin{aligned}\label{eq:M-cuts}
           \prod_{j \in S} x_j^{(1)} = \prod_{j \in S} x_j^{(2)},\forall S \in \mathcal{S}.
        \end{aligned}
    \end{equation}
    \item \emph{Vertex-based  reformulation}: The primal-redundant constraints take the form
    \begin{equation}
        \begin{aligned} \label{eq:V-cuts}
           \prod_{j \in [n]} \sigma_{v_j}(x_j^{(1)}) = \prod_{j \in [n]} \sigma_{v_j}(x_j^{(2)}),\text{ for each vertex } \vv \text{ of } [0,1]^{n},
        \end{aligned}
    \end{equation}
    where $\sigma_{v_j}(u) := \begin{cases}
        u & \text{ if } v_j = 1 \\
        1-u & \text{ if } v_j = 0
    \end{cases}$.
\end{enumerate}
We refer to the Lagrangian dual of monomial-based reformulation and vertex-based reformulation as \textit{M-Lagrangian dual} and \textit{V-Lagrangian dual} respectively.
\begin{remark}{Difference between traditional RLT-type constraints and (\ref{eq:M-cuts})}: For the constraints of the form $x^{(1)}_j = x^{(2)}_j$ for say $j \in \{1,2\}$, traditional RLT would involve multiplying these constraints, that is, adding new constraints of the form $\left(x^{(1)}_1 - x^{(2)}_1\right)\left(x^{(1)}_2 - x^{(2)}_2\right)=0$ in the reformulation step of RLT.  Such a constraint would not be separable after dualizing. Constraints of the form of (\ref{eq:M-cuts}) are more similar to those used in the reformulation in~\cite{burer2009copositive,cifuentes2024sensitivity}. Also see~\cite{minoux2010drl}.    
\end{remark}

We note here that the monomial-based reformulation and the vertex-based reformulation are quite different. 
As we will see in Section~\ref{sec:M2}, one has to add (\ref{eq:M-cuts}) corresponding to all monomials to obtain zero duality gap in general. In practice, in Section~\ref{sec:compute}, we will see significant dual bound improvement over the classical Lagrangian dual with just the quadratic monomials. 

On the other hand, V-Lagrangian dual is more like a ``column-generation" approach. For strong duality to be attained, we need to add (\ref{eq:V-cuts}) for all vertices of $[0, 1]^n$ (Section~\ref{sec:v-converge}). We will show that the inner problem (\ref{ex_dual_prob_inner}) in this case can be emulated in original space $\mathcal{X}^{(i)}$ using ``no-good cuts"~\cite{angulo2016improving}. The number of extra cuts is initially zero and increases with each iteration,  until (\ref{ex_dual_prob}) is solved. The resulting method is similar to the algorithm described in \cite{ahmed2013scenario}, where adding weights to vertices is replaced by cutting off the vertices. We will show in Section~\ref{sec:v-converge} that sub-gradient method for V-Lagrangian dual has a convergence rate that is at least as good as that of~\cite{ahmed2013scenario}.


\subsection{M-Lagrangian dual for two-block problems.}\label{sec:M2}

Fix any $\mathcal{S} \subseteq 2^{[n]}$, we let
\begin{align*}
        \mathcal{X}^{(i)}_{M}(\mathcal{S}) := \left\{(\xv^{(i)},\yv^{(i)},\wv^{(i)}) : (\xv^{(i)},\yv^{(i)}) \in \mathcal{X}^{(i)}, w^{(i)}_S = \prod_{j \in S} x_j^{(i)},\forall S \in \mathcal{S}\right\}.
    \end{align*}
We assume that $\{j\} \in \mathcal{S},\forall j \in [n]$. 
In this case, (\ref{ex_main_prob_link}) is included within (\ref{ex_main_prob_link_extra}) and so the monomial reformulation 
takes the form:
\begin{subequations}
\label{ex_main_prob_monomial}
\begin{align}
\OPT =  \min_{(\xv,\yv,\wv)} \ & \sum_{i \in \{1,2\}}  \iprod{\cv^{(i)}}{\xv^{(i)}} + \iprod{\dv^{(i)}}{\yv^{(i)}} \\
\text{s.t.}\ & (\xv^{(i)},\yv^{(i)},\wv^{(i)}) \in \mathcal{X}^{(i)}_{M}(\mathcal{S}),\forall i \in \{1,2\}, \\
\label{ex_main_prob_monomial_link_cons}
& \ \wv^{(1)}_S = \wv^{(2)}_S,\forall S \in \mathcal{S}.
\end{align}
\end{subequations}
If we apply Lagrangian decomposition to (\ref{ex_main_prob_monomial}) by dualizing (\ref{ex_main_prob_monomial_link_cons}), then by Geoffrion's Theorem (also called primal characterization~\cite{boland2018combining}) the optimal value of the M-Lagrangian dual is equal to: 
\begin{equation}
\label{ex_main_prob_monomial_primal}
\begin{aligned}
\LAG^{M} =  \min_{(\xv,\yv,\wv)} \ & \sum_{i \in \{1,2\}}  \iprod{\cv^{(i)}}{\xv^{(i)}} + \iprod{\dv^{(i)}}{\yv^{(i)}} \\
\text{s.t.}\ & (\xv^{(i)},\yv^{(i)},\wv^{(i)}) \in \conv\{\mathcal{X}^{(i)}_{M}(\mathcal{S})\},\forall i \in \{1,2\}, \\
& \ \wv^{(1)}_S = \wv^{(2)}_S,\forall S \in \mathcal{S}.
\end{aligned}
\end{equation}
Throughout the section, we make the following assumption:
\begin{assumption}
    \label{asp:down_close_set}
     $\mathcal{S}$ is down-closed. Namely, if $S \in \mathcal{S}$, then $S' \in \mathcal{S}$ for all $S' \subseteq S$.
\end{assumption}

As the objective function of (\ref{ex_main_prob_monomial_primal}) only involves $(\xv,\yv)$, we are interested in the projection of the feasible region of (\ref{ex_main_prob_monomial_primal}) onto $(\xv,\yv)$:
\begin{equation}
    \begin{aligned}
        \mathcal{A}(\mathcal{S}) := \left\{ (\xv,\yv) \;\middle\vert\;
   \begin{array}{@{}l@{}} \exists \wv,
 (\xv^{(i)},\yv^{(i)},\wv^{(i)}) \in \conv\{\mathcal{X}^{(i)}_{M}(\mathcal{S})\},\forall i \in \{1,2\}, \\ 
\wv^{(1)}_S = \wv^{(2)}_S,\forall S \in \mathcal{S}
   \end{array} 
\right\}.
    \end{aligned}
\end{equation}


We further define
\begin{equation}
    \begin{aligned}
        \mathcal{B}(\mathcal{S}) := \bigcap_{\mathcal{U} \in  \mathcal{S}}  \conv \left\{ (\xv,\yv) \;\middle\vert\;
   \begin{array}{@{}l@{}} 
 (\xv^{(i)},\yv^{(i)}) \in \mathcal{X}^{(i)},\forall i \in \{1,2\}, \\
x_j^{(1)} = x_j^{(2)},\forall j \in \mathcal{U}
   \end{array}
\right\}.
\end{aligned}
\end{equation}

\begin{theorem}
    \label{thm:A_subset_B}
    Under Assumption \ref{asp:down_close_set} on $\mathcal{S}$, we have $\mathcal{A}(\CS) \subseteq \mathcal{B}(\CS)$. As a consequence, when $\mathcal{S} = 2^{[n]}$, strong duality holds between (\ref{ex_main_prob_monomial}) and (\ref{ex_main_prob_monomial_primal}).
\end{theorem}

Details of the proof are presented in the Appendix \ref{sec:M-Lagrangian for multi-block}.


\subsection{Bounds for two-block packing and covering problem.}\label{sec:M2PC}

In this subsection, we consider the $M$-Lagrangian dual when $\mathcal{S} =  \binom{[n]}{\leq k}$ for some fixed number $k \in [n]$. We refer to the corresponding dual optimal value by $\LAG^{M}_k$.
The goal of this subsection is to present a multiplicative bound on $\LAG^{M}_k$ for packing and covering instances. 
Since we use a ``min" objective and ``$\leq$" constraints,  (\ref{prob}) is called \textit{packing} (resp. \emph{covering}) if each $\cv^{(i)}$, $\dv^{(i)}$ are non-positive (resp. non-negative) and $A^{(i)}$, $B^{(i)}$, $\bv^{(i)}$ are entry-wise non-negative (resp. non-positive). Therefore, in our notation, $\OPT$ and $\LAG^M_k$ are non-positive (resp. non-negative) for packing (resp. covering) problems.


For packing problems, we require the following assumption for our analysis, which is commonly made in the stochastic programming literature (see, for example~\cite{chen2022sample}~\cite{higle2005stochastic}~\cite{rockafellar1976stochastic}).

\begin{assumption}
  \label{asp_recourse}
  (Relatively complete recourse)
  For any $(\xv^{(1)},\yv^{(1)}) \in \mathcal{X}^{(1)}$, there exists some $\yv^{(2)}$ that $(\xv^{(1)},\yv^{(2)}) \in \mathcal{X}^{(2)}$. Symmetrically, for any $(\xv^{(2)},\yv^{(2)}) \in \mathcal{X}^{(2)}$, there exists some $\yv^{(1)}$ that $(\xv^{(2)},\yv^{(1)}) \in \mathcal{X}^{(1)}$.
\end{assumption}

Note that if this assumption is violated, then one can simply add $A^{(2)} \xv^{(1)} \leq \bv^{(2)}$ and $A^{(1)} \xv^{(2)} \leq \bv^{(1)}$ into $\CX^{(1)}$ and $\CX^{(2)}$ respectively and it is straightforward to see that this assumption is then satisfied for packing problems. 




\begin{theorem}
\label{thm_two_block_packing_bound} For any packing instance (\ref{prob}), under Assumption \ref{asp_recourse}, we have that $\left(2+\dfrac{1}{t-2}\right) \cdot \OPT \leq \LAG_k^{M} \leq \OPT$, where $t = k/n.$
\end{theorem}

\begin{theorem}
\label{thm_two_block_covering_bound}
For any covering instance (\ref{prob}), we have $\left(\dfrac{1}{2-t}\right)\cdot \OPT \leq \LAG_k^{M} \leq \OPT$ where $t = k/n$.
\end{theorem}


The proof presented in Appendix \ref{sec:M-Lagrangian for multi-block} of the above results uses techniques similar to those presented in~\cite{dey2018analysis}. 

\subsection{V-Lagrangian dual for two-block problem.}\label{sec:v-converge}

\begin{align*}
     \textup{Let }\mathcal{X}^{(i)}_{V} := \left\{ (\xv^{(i)},\yv^{(i)},\wv^{(i)}) \;\middle\vert\;
   \begin{array}{@{}l@{}}
 (\xv^{(i)},\yv^{(i)},\wv^{(i)}) \in \mathcal{X}^{(i)}, \\ 
w^{(i)}_{\vv} = \prod_{s \in [n]} \sigma_{v_s}((\xv^{(i)})_s),\text{$\forall \vv$ vertex of $[0,1]^n$} \}
   \end{array} 
\right\}.
\end{align*}
By replacing $\mathcal{X}^{(i)}_{V}$ in place of $\mathcal{X}^{(i)}_{ex}$ in (\ref{ex_main_prob}) and  dualizing, we obtain:


\begin{equation}
    \begin{aligned}
        \label{ex_dual_prob_vertex_inner}
        L^{V}(\lambda,\mu) := \min_{(\xv,\yv,\wv)} & \left(\sum_{i \in \{1,2\}}  \iprod{\cv^{(i)}}{\xv^{(i)}} + \iprod{\dv^{(i)}}{\yv^{(i)}}\right) \\ & +   \iprod{\lambda}{\xv^{(1)} - \xv^{(2)}}  + \iprod{\mu}{\wv^{(1)} - \wv^{(2)}}\\
        \ & \text{s.t.}\ (\xv^{(i)},\yv^{(i)},\wv^{(i)}) \in \mathcal{X}^{(i)}_{V},\forall i \in \{1,2\},
    \end{aligned}
\end{equation}
and
\begin{equation}
    \begin{aligned}
        \label{ex_dual_prob_vertex}
        \LAG^{V} := \max_{\lambda,\mu}  L^{V}(\lambda,\mu).
    \end{aligned}
\end{equation}
\begin{theorem}
    \label{thm_strong_dual_vertex} Strong duality holds for V-Lagrangian dual, that is 
    $\OPT = \LAG^V$.
\end{theorem}
Details of the proof are presented in the Appendix \ref{sec:Strong duality of V-Lagrangian}.
Now we show how to solve (\ref{ex_dual_prob_vertex}) in the original space. The inner problem (\ref{ex_dual_prob_vertex_inner}) is decomposed into sub-problems taking the following form for some values of $\lambda, \mu$:
\begin{equation}
       \label{vertex_dual_prob_inner}
        \min_{\xv^{(i)},\yv^{(i)}} \left\{\iprod{\cv^{(i)} + \lambda}{\xv^{(i)}} + \iprod{\dv^{(i)}}{\yv^{(i)}} + \iprod{\mu}{\wv^{(i)}} : (\xv^{(i)},\yv^{(i)},\wv^{(i)}) \in \mathcal{X}^{(i)}_V\right\}.
\end{equation}


Given $\mathcal{V} \subseteq \{0,1\}^n$, one can model $\xv \in \{0,1\}^{n} \setminus \mathcal{V}$ by no-good cuts or more general methods~\cite{laporte1993integer,angulo2015forbidden}.
Thus, solving (\ref{vertex_dual_prob_inner}) can be accomplished in the original space by dividing the optimization task into two parts as shown in Algorithm~\ref{solve_vertex_dual_prob_inner}.

\begin{algorithm}[H]
  \caption{Solve (\ref{vertex_dual_prob_inner})}\label{solve_vertex_dual_prob_inner}
  \Input{$i,\lambda,\mu$}
  \Output{the optimal solution of (\ref{vertex_dual_prob_inner})}

$\mathcal{V} \leftarrow \{\vv : \mu_{\vv} \neq 0 \}$
  
   $(\xv_*,\yv_*) \leftarrow \argmin\limits_{(\xv^{(i)},\yv^{(i)}) \in \mathcal{X}^{(i)}} \iprod{\cv^{(i)} + \lambda}{\xv^{(i)}} + \iprod{\dv^{(i)}}{\yv^{(i)}} : \xv^{(i)} \not\in \mathcal{V}$

   $p_* \leftarrow \iprod{\cv^{(i)} + \lambda}{\xv_*} + \iprod{\dv^{(i)}}{\yv_*}$

  \ForEach{$\vv \in \mathcal{V}$ }{
  $(\vv,\yv') \leftarrow \argmin\limits_{(\vv,\yv^{(i)}) \in \mathcal{X}^{(i)}} \iprod{\cv^{(i)} + \lambda}{\vv} + \iprod{\dv^{(i)}}{\yv^{(i)}} $

    $p' \leftarrow \iprod{\cv^{(i)} + \lambda}{\vv} + \iprod{\dv^{(i)}}{\yv'} + \mu_{\vv} $

  \If{$p' < p_*$}{
    $\xv_* \leftarrow \vv,\yv_* \leftarrow \yv'$
    $p_* \leftarrow p'$
  }
  }
  \Return $\xv^*,\yv_*,p_*$
\end{algorithm} 


  

Algorithm \ref{solve_vertex_dual_prob_inner} uses the fact that any feasible solution $(\wv^{(i)})_{\vv}$ is non-zero if and only if $\xv^{(i)} = \vv$. Another consequence of this fact is that, since there are two sub-problems of the type (\ref{vertex_dual_prob_inner}), the support of the sub-gradient (on the $\mu$ variables) is at most $2$.
Therefore, if one updates $\mu,\lambda$ by the sub-gradient method or classical bundle(-level) method, then $\norm{\mu}_0$ increases by at most $2$ for each iteration. Also, observe that     $\argmin\limits_{(\vv,\yv^{(i)}) \in \mathcal{X}^{(i)}} \iprod{\cv^{(i)} + \lambda}{\vv} + \iprod{\dv^{(i)}}{\yv^{(i)}}$ is independent of $\lambda,\mu$ and therefore one may store the choice of $\yv'$ for each $\vv$ to avoid repeated evaluation in Algorithm 1.

    





We next compare the number of iterations to solve (\ref{ex_dual_prob_vertex})
using a sub-gradient method  
against the number of iterations for the scenario decomposition algorithm (SDA) of~\cite{ahmed2013scenario}. 
At a high level, the main difference between our approach here and the 
SDA is that instead of updating the dual values corresponding to the solutions found in previous iteration, the SDA simply eliminates the corresponding solutions from each sub-problem in the next iteration.  Details of SDA are presented in the Appendix~\ref{sec:SDA}. In particular,  we show that using a sub-gradient method to solve (\ref{ex_dual_prob_vertex}) cannot be much worse than SDA with $\lambda =0$ setting (this setting is the main focus of~\cite{ahmed2013scenario}) under the assumption of non-negativity of objective function value. 
\begin{proposition}
    \label{prop_comparision} 
    Suppose that $c^{(i)}, d^{(i)}$ are non-negative for $i \in \{1,2\}$. Suppose that SDA with $\lambda = 0$ terminates in $t_{SDA}$ iterations. 
    Let $\phi := \max_{i \in \{1,2\},(\xv^{(i)},\yv^{(i)}) \in \mathcal{X}^{(i)}}$ $ \left\vert\iprod{\cv^{(i)}}{\xv^{(i)}} + 
   \iprod{\dv^{(i)}}{\yv^{(i)}} \right\vert$. Then there exists a sub-gradient algorithm that solves (\ref{ex_dual_prob_vertex}) to $\epsilon$-optimality (additive error) within $O\left( \left(\frac{  \sqrt{n}\phi}{\epsilon}\right)^2\right)\cdot t_{SDA} \quad { iterations.}$
\end{proposition}

We also show that there are simple instances where SDA can be exponentially worse than a sub-gradient algorithm to solve (\ref{ex_dual_prob_vertex}). Details of the proof of Proposition~\ref{prop_comparision} and Proposition~\ref{prop:SDA_bad} are presented in Appendix \ref{sec:Comparison between V-Lagrangian and SDA}.
\begin{proposition}
\label{prop:SDA_bad}There exists an instance where SDA with $\lambda = 0$ setting requires $O\left(2^n\right)$ iterations to achieve $\epsilon$-optimality, while (\ref{ex_dual_prob_vertex}) can be solved to $\epsilon$-optimality by classic sub-gradient method in $O\left({n^2 (n+2)}\right)$ iterations.
\end{proposition}



\section{Lagrangian dual for general MIPs.}\label{sec:general}

In this section, we aim to generalize the ideas applied to the two-block case to the case of general MIPs, which can be thought of as multiple blocks of constraints and each pair of blocks of constraints share some common variables. 

\paragraph{V-Lagrangian dual.} We first note that the Algorithm~\ref{solve_vertex_dual_prob_inner} that allows for conducting  the computations in the original space, does not generalize very easily to the case of multiple-blocks of constraints, where different pairs of blocks share different sets of common variables. Therefore, the most natural generalization that can be tackled by the V-Lagrangian  dual is the case of two-stage stochastic programming type instances (also called star instances) where the blocks are: 
 (i) the constraints involving the first stage variables only, (ii) the constraints corresponding to a particular scenario in second stage. 
If the same set of first-stage binary variables appears in the constraints for every scenario, then it is easy to generalize Algorithm~\ref{solve_vertex_dual_prob_inner} in this setting, and it is straightforward to see that the V-Lagrangian dual has zero duality gap. 

\paragraph{M-Lagrangian for multi-block problems.}
For the case of general MIPs, we only pursue the M-Lagrangian dual approach. We first observe that strong duality may fail when naively constructing the M-Lagrangian dual, even if all constraints corresponding to all monomials are included. Consider the 3-block example:
\begin{subequations}
    \label{three_block_example}
    \begin{align}
        \min \ & t^{(1)} + t^{(2)} + t^{(3)} \\
       \text{s.t. } & (\xv^{(1)},t^{(1)}) \in \{ (\xv,t) :\xv \in \{0,1\}^2, t = x_1 x_2\}, \\
        & (\xv^{(2)},t^{(2)}) \in \{ (\xv,t) :\xv \in \{0,1\}^2, t =(1 - x_1) x_2\}, \\
        & (\xv^{(3)},t^{(3)}) \in \{ (\xv,t) :\xv \in \{0,1\}^2, t =(1 - x_1) (1-x_2)\}, \\
        \label{three_block_example_link}
        & x^{(1)}_1 = x^{(2)}_1,  x^{(1)}_2 = x^{(3)}_1, x^{(2)}_2 = x^{(3)}_2.
    \end{align}
\end{subequations}
\label{rem:strong_dual_fail}
The optimal value (\ref{three_block_example}) is $1$. On the other hand, each coupling constraint between any two blocks in (\ref{three_block_example_link}) only involves one variable. Therefore,
the classical Lagrangian of (\ref{three_block_example}) by dualizing (\ref{three_block_example_link}) coincides with M-Lagrangian with all monomials. The optimal value of classical Lagrangian dual is $0$ for this example.

Example (\ref{three_block_example}) illustrates that strong duality may not be achieved when naively applying the reformulations discussed in Section~\ref{sec:MV}. Such issue arises because the coupling constraints (\ref{three_block_example_link}) form a cycle and the generalization of Theorem \ref{thm:A_subset_B} does not go through. Therefore, it is natural to explore the tree-decomposition approach.

Consider a general MIP:
\begin{equation}
    \label{general_mip}
    \min \{\cv^{\top} \xv: A \xv \leq \bv, 
       \xv \text{ is  mixed-binary} \},
\end{equation}
where $\cv,\bv,A$ are all rational. 
\begin{definition} 
   The intersection graph \cite{fulkerson1965incidence} of (\ref{general_mip}) is a simple undirected graph that has a vertex for each variable in (\ref{general_mip}) and two vertices are adjacent if and only if their associated variables appear in any common constraint of $Ax \leq b$.
\end{definition}

Using the intersection graph of (\ref{general_mip}), one can reformulate (\ref{general_mip}) into ``a tree-structure" via the \textit{tree decomposition} \cite{robertson1986graph}.

\begin{definition}
    Let $G$ be a simple undirected graph. A tree decomposition of $G$ is a pair of $(\mathcal{T},\mathcal{Q})$ where $T$ is a tree and $\mathcal{Q} = \{ \mathcal{Q}_t : t \in V(\mathcal{T}) \}$ is a collection of vertices of $V(G)$ such that the following holds:
    \begin{enumerate}
        \item For each $v \in V(G)$, the set $\{t \in V(\mathcal{\mathcal{T}}) : v \in \mathcal{Q}_t \}$ forms a subtree of $\mathcal{T}$,
        \item If $(u,v) \in E(G)$, then there exists $t \in V(\mathcal{T})$ such that $u,v \in \mathcal{Q}_t$,
        \item $\bigcup_{t \in V(\mathcal{T})} \mathcal{Q}_t = V(G)$.
    \end{enumerate}
\end{definition}

Given a general MIP, its intersection graph $\mathcal{G}$, and a corresponding tree decomposition $(\mathcal{T},\mathcal{Q})$, one can reformulate (\ref{general_mip}) as:
\begin{equation}
    \begin{aligned}
        \label{tree_prob}
        \min \ & \sum_{i \in V(\mathcal{T})} \iprod{\cv^{(i)}}{\xv^{(i)}} \\
        \text{s.t. } &  \xv^{(i)} \in \mathcal{X}^{(i)},\forall i \in V(\mathcal{T}), \ x^{(i)}_v = x^{(j)}_v,\forall (i,j) \in E(\mathcal{T}) \text{ and } v \in \CQ_i \cap \CQ_j,
    \end{aligned}
\end{equation}
%
where each $\mathcal{X}^{(i)}$ makes local copies of variables in $\CQ_i$ and
each constraint in $A \xv \leq \bv$ is absorbed in some $\CX^{(i)}$  using a copy of the variables  
in $\CX^{(i)}$.
Henceforth, we assume that coupling constraints involve only binary variables. 
\subsection{Strong duality for M-Lagrangian dual for multi-block problems.}
For each $i \in V(\mathcal{T})$, we define
\begin{align*}
    \CX^{(i)}_{M} := \left\{ (\xv^{(i)},\wv^{(i)}) : \xv^{(i)} \in \CX^{(i)}, w^{(i)}_{S} = \prod_{i \in S} x^{(i)}_S,\forall S \subseteq \CQ_i \right\}.
\end{align*}
Let $\CQ_{ij} \subseteq V(G)$ denote $\CQ_i \cap \CQ_j$. For every edge $(i,j) \in E(\mathcal{T})$, let $\mathcal{S}_{ij} \subseteq 2^{\mathcal{Q}_{ij}}$ be a collection of subsets of $\CQ_{ij}$.
The extended formulation of (\ref{tree_prob}) parameterized by 
$\mathcal{S}:=\{\mathcal{S}_{ij}\}_{(i,j) \in E(\mathcal{T})}$%
takes the form of
\begin{equation}
    \begin{aligned}
        \label{ex_tree_prob}
        \min \ & \sum_{i \in V(\mathcal{T})} \iprod{\cv^{(i)}}{\xv^{(i)}} \\
        \text{s.t. } &  (\xv^{(i)},\wv^{(i)}) \in \mathcal{X}^{(i)}_M,\forall i \in V(\mathcal{T}),\   \wv^{(i)}_S = \wv^{(j)}_S, \forall (i,j) \in E(\mathcal{T}), S \in \mathcal{S}_{ij}.
    \end{aligned}
\end{equation}
\begin{assumption}
    \label{asp:tree_down_close}
    $\mathcal{S}_{ij}$ is down-closed for all $(i,j) \in E(\mathcal{T})$.
\end{assumption}
Similar to arguments in Section~\ref{sec:M2}, we consider the projection of the feasible region of the primal characterization of M-Lagrangian dual onto $\xv$:
\begin{equation}
    \begin{aligned}
        \mathcal{A}({\mathcal{S}},\mathcal{T}) := \left\{ \xv \;\middle\vert\;
   \begin{array}{@{}l@{}} \exists \wv,
 (\xv^{(i)},\wv^{(i)}) \in \conv\left\{\mathcal{X}^{(i)}_M\right\},\forall i \in V(\mathcal{T}), \\ 
 \wv^{(i)}_S = \wv^{(j)}_S, \forall (i,j) \in E(\mathcal{T}), S \in \mathcal{S}_{ij}
   \end{array} 
\right\},
    \end{aligned}
\end{equation}
and
\begin{equation}
    \begin{aligned}
        \mathcal{B}({\mathcal{S}},\mathcal{T}) := \bigcap_{\mathcal{U} \unlhd \mathcal{S}} \conv \left\{
   \begin{array}{@{}l@{}} 
 \xv^{(i)} \in \mathcal{X}^{(i)},\forall i \in V(\mathcal{T}), \\
         x^{(i)}_v = x^{(j)}_v,\forall (i,j) \in E(\mathcal{T}),\forall v \in \mathcal{U}_{ij}.
   \end{array}
\right\},
    \end{aligned}
\end{equation}

where $\mathcal{U} \unlhd\mathcal{S}$ implies $\mathcal{U} = \{\mathcal{U}_{ij}\}_{(i,j) \in E(\mathcal{T})}$ and $\mathcal{U}_{ij} \in \mathcal{S}_{ij}$ for all $(i,j) \in \mathcal{T}$.
\begin{theorem}
    \label{thm:strong_duality_tree}
    Under the assumption \ref{asp:tree_down_close}, $\mathcal{A}({\mathcal{S}},\mathcal{T}) \subseteq \mathcal{B}({\mathcal{S}},\mathcal{T})$. Therefore, when $\CS_{ij} = 2^{\CQ_{ij}},\forall (i,j) \in E(\mathcal{T})$, then strong duality holds between (\ref{tree_prob}) and its M-Lagrangian relaxation.
\end{theorem}
Proof of Theorem \ref{thm:strong_duality_tree} is presented in the Appendix \ref{sec:M-Lagrangian for multi-block}.

\subsection{Hierarchical bounds for multi-block packing and covering problems}

In this subsection, we consider the $M$-Lagrangian  dual when $\mathcal{S}_{ij} =  \binom{\CQ_{ij}}{\leq k}$ for some fixed number $k \in [n]$. We refer  to $\LAG^{M},\mathcal{A}(\mathcal{S},\CT),\mathcal{B}(\mathcal{S},\CT)$ in this setting by $\LAG^{M}_k,\mathcal{A}_k(\CT),\mathcal{B}_k(\CT)$. The goal of this subsection is again to analyze $\LAG^{M}_k$ as a function of $k$ for packing and covering instances through the lens of bound implied by $\mathcal{B}_k$. 
Similar to assumption \ref{asp_recourse} in the two-block case, we need an additional assumption to analyze packing problems:
\begin{assumption}
  \label{asp_local_recourse}
  (Relatively complete recourse - general)
  For any $i \in V(\CT)$ and any $\xv^{(i)} \in \mathcal{X}^{(i)}$, there exists a feasible solution $\xv_*$ in (\ref{tree_prob}) such that $\xv_*^{(i)} = \xv^{(i)}$.
\end{assumption}

Like before, this assumption is easy to achieve by addition of suitable constraints. 
To present our result, we require the following definitions. 
\begin{definition}[Good and k-good]
    Given a subset $\mathcal{W}$ of variables $\xv$ in (\ref{general_mip}), let $\mathcal{V} := \{ i \in V(\CT) : \CQ_i \cap  \mathcal{W} \neq \emptyset\}$. Consider the sub-graph $\CT(\mathcal{W})$ of $\CT$ induced by $\mathcal{V}$. We say $\mathcal{W}$ is \textit{good} if every connected component $\mathcal{C}$ of $\CT(\mathcal{W})$ satisfies either
    \begin{enumerate}
        \item For any $(i,j) \in E(\mathcal{C}), |\CQ_i \cap \CQ_j \cap \mathcal{W}| \leq k$.
        \item There exists $i \in V(\mathcal{C})$ such that $(\CQ_j \cap \mathcal{W}) \subseteq (\CQ_i \cap \mathcal{W}),\forall j \in V(\mathcal{C})$.
    \end{enumerate}
    When every connected component $\mathcal{C}$ of $\CT(\mathcal{W})$ satisfies (1.), we call $\mathcal{W}$ $k$-good.
\end{definition}

For a subset $\mathcal{W}$ of variables in (\ref{general_mip}), we let $\chi_{\mathcal{W}}$ to be the indicator vector of $\mathcal{W}$, that is 
we let $(\chi_{\mathcal{W}})_{x} = 1$ if $ x \in \mathcal{W}$ and $0$ otherwise. 
Consider the following two linear programs:

\begin{equation}    \label{multi_block_packing_bound}
\eta_k := \left\{\min \sum_{\mathcal{W} \text{ is good}} \alpha_{\mathcal{W}}: \sum_{\mathcal{W} \text{ is good}} \alpha_{\mathcal{W}} \chi_{\mathcal{W}} \geq \textbf{1} \text{ and }\alpha_{\mathcal{W}} \geq 0.\right\}
\end{equation}


\begin{equation}
\label{multi_block_covering_bound}
\theta_k := \left\{\min \sum_{\mathcal{W} \text{ is $k$-good} } \alpha_{\mathcal{W}} : \sum_{ \mathcal{W} \text{ is $k$-good} } \alpha_{\mathcal{W}} \chi_{\mathcal{W}} \geq \textbf{1} \text{ and }\alpha_{\mathcal{W}} \geq 0.
    \right\}
\end{equation}


\begin{theorem}
    \label{thm_Hierarchical_bound_for_packing}
For any packing instance, under Assumption \ref{asp_local_recourse}, we have that $$\eta_k \cdot\OPT \leq \LAG_k^{M} \leq \OPT.$$
\end{theorem}

\begin{theorem}
    \label{thm_Hierarchical_bound_for_covering}
For any covering instance (\ref{general_mip}), let $\tau := \max_{v \in V(\mathcal{G})} |\{i \in V(\mathcal{T}): v \in \CQ_i \}| $. Then we have that $$\dfrac{\theta_k}{ 1-\tau + \tau \cdot \theta_k}\cdot \OPT \leq \LAG_k^{M} \leq \OPT.$$
\end{theorem}
In fact, Theorem~\ref{thm_Hierarchical_bound_for_packing} and Theorem~\ref{thm_Hierarchical_bound_for_covering} are generalizations of respectively  Theorem~\ref{thm_two_block_packing_bound} and Theorem~\ref{thm_two_block_covering_bound}, whose proofs can be found in the Appendix \ref{sec:M-Lagrangian for multi-block}.  

\begin{corollary}
    \label{cor:application_for_two_stage}
    Let $t := \dfrac{k}{n}$.
    For two-stage packing problem with $Z$ scenarios, under Assumption \ref{asp_local_recourse}, we have $$\left(2+\dfrac{2t-Z\cdot t-1}{Z-t}\right) \cdot \OPT \leq \LAG_k^{M} \leq \OPT.$$ For two-stage covering problem with $Z$ scenarios, we have $$\dfrac{1}{(1-Z)\cdot t + Z} \cdot \OPT \leq \LAG_k^{M} \leq \OPT.$$
\end{corollary}

\begin{remark}
    \label{rem:tightness}
    The proof of Theorem \ref{thm_Hierarchical_bound_for_packing} and Theorem \ref{thm_Hierarchical_bound_for_covering} depends on Theorem \ref{thm:strong_duality_tree}. Both Theorems in fact present multiplicative bounds on optimizing over $\mathcal{B}(\mathcal{T})_k$ as against bounds directly for (\ref{tree_prob}). 
    Since Theorem \ref{thm:strong_duality_tree} states that $\mathcal{A}(\mathcal{T})_k \subseteq \mathcal{B}(\mathcal{T})_k$ and $\LAG_k^M$ is obtained by optimizing over $\mathcal{A}(\mathcal{\mathcal{T}})_k$, we don't expect that the multiplicative bounds presented in Theorem \ref{thm_Hierarchical_bound_for_packing} and Theorem \ref{thm_Hierarchical_bound_for_covering} are tight for general choice of $k$. When $k = 1$, which corresponds to the classic Lagrangian dual, one can show that $\mathcal{A}(\mathcal{T})_1 = \mathcal{B}(\mathcal{T})_1$ and in this case, \cite{dey2018analysis} provides examples where Theorem \ref{thm_Hierarchical_bound_for_packing} and Theorem \ref{thm_Hierarchical_bound_for_covering} are asymptotically tight as $n \to \infty$. For two block case, when $n = 2, k = 1$, the worst packing instance we find satisfies that $\dfrac{\LAG^M_k}{\OPT} \geq \dfrac{5}{4}$ while Theorem \ref{thm_Hierarchical_bound_for_packing} indicates $\dfrac{\LAG^M_k}{\OPT} \leq \dfrac{4}{3}$. When $n = 2, k = 1$, the worst covering instance we find satisfies that $\dfrac{\LAG^M_k}{\OPT} \leq \dfrac{3}{4}$ while Theorem \ref{thm_Hierarchical_bound_for_covering} indicates $\dfrac{\LAG^M_k}{\OPT} \geq \dfrac{2}{3}$. It remains open that whether the bound provided by Theorem \ref{thm_Hierarchical_bound_for_packing} and Theorem \ref{thm_Hierarchical_bound_for_covering} are tight. 
\end{remark}


\section{Preliminary numerical experimental}\label{sec:compute}

Our experiments are run on a Windows PC with a 12th Gen Intel(R) Core(TM) i7 processor and 16 GB RAM. We use Julia 1.9, relying on Gurobi version 9.0.2 as MIP solver and Clarabel~\cite{goulart2024clarabel} as QP solver (for bundle method implementation).

We generated two classes of instances which we refer to as (STAR-STAB) and (PATH-STAB). For both problems we construct $10$ blocks, each block solving a maximum cardinality stable set problem on a random graph on $100$ nodes. 
For (STAR-STAB) the blocks are coupled by $
\xv^{(1)}_j = \xv^{(i)}_j$, $\forall i \in \{2,\dots,10\}$, $\forall j \in \{1,\dots,33\}.$ For (PATH-STAB) the blocks are coupled by $
\xv^{(i)}_{67+j} = \xv^{(i+1)}_j,\forall i \in \{1,\dots,9\},\forall j \in \{1,\dots,33\}.
$
For each class of problems, we generate 10 instances.

For (PATH-STAB) we consider classic Lagrangian dual (L) and M-Lagrangian dual with only quadratic terms (QL). For (STAR-STAB), we additionally consider V-Lagrangian dual (VL) and SDA \cite{ahmed2013scenario}. We solve the Lagrangian dual by bundle method. For all Lagrangian duals we first start by solving classical Lagrangian dual for 600 seconds. Then we take this dual solution, append zeros for the additional dual variables and use this solution as the starting point for solving the more sophisticated Lagrangian dual in the remaining time. In our naive implementation, we solved the sub-problems sequentially.
See details in the Appendix \ref{sec:details}.

We set 1200 seconds as time limit for all methods. In our experiments, none of the methods terminated. In particular, the initial upper bound required by bundle methods was derived using the feasible solution of Gurobi. During the course of the run of the bundle methods, these upper bounds did not change. 


{\footnotesize
\begin{table}[h!]\centering
\caption{STAR-STAB}
\label{tab:comb}
\setlength{\tabcolsep}{3pt}
\begin{tabular}{@{}cccccc@{}}\toprule 
Methods
& Primal-dual gap   & Time  &  Iterations \\
\midrule
Gurobi & 10.0\%    &  1200 &-\\
L & 6.0\%    & 1200 & 127  \\
QL & 4.4\%    & 1200 &  96 \\
VL & 3.9\%     & 1200 & 103  \\
SDA & 8.6\%     & 1200 & 151  \\
\bottomrule
\end{tabular}
\caption{PATH-STAB}\label{tab:comb1}
\begin{tabular}{@{}cccccc@{}}\toprule 
Methods
& Primal-dual gap    & Time  &  Iterations \\
\midrule
Gurobi & 10.8\%   &  1200 &-\\
L & 3.5\%    & 1200 & 369  \\
QL & 1.2\%   & 1200 & 254  \\
\bottomrule
\end{tabular}
\end{table}
}



The results of our computation are present in Table~\ref{tab:comb}. The primal-dual gap is computed as 1 minus the ratio of the dual bound from the Lagrangian relaxation and the primal bound given by Gurobi's feasible solution. 
We observe that in the case of (STAR-STAB), the V-Lagrangian dual has the best performance, closely followed by M-Lagrangian dual with quadratic terms. The dual bounds obtained from classical Lagrangian dual and SDA are  worse. For the (PATH-STAB) instances (Table~\ref{tab:comb1}), the M-Lagrangian dual with quadratic terms outperforms the classical Lagrangian dual. For both instance classes, Gurobi's performance in obtaining dual bounds is significantly worse as it does not exploit decomposability. 

\section{Concluding remarks}\label{sec:conclude} We show that constructing  Lagrangian duals which achieve the twin goal of zero duality gap and maintaining decomposability of the sub-problems, comes at the cost of solving more challenging sub-problems with non-linear objective functions. Thus, this approach is perhaps most suitable when we have a large number of sub-problems, but each subproblem is not too large. An open question is whether one can achieve Lagrangian duals with decomposability and zero duality gap while solving an "easier" subproblem in each iteration.

On the computational side, we see that even implementing the sub-problems sequentially leads to significant improvements over both Gurobi (that cannot exploit decomposability) or the classical Lagrangian dual. We caution the reader that there are significantly more sophisticated methods for star (2-stage stochastic) problems~\cite{kuccukyavuz2017introduction}. While our methods may not necessarily be competitive against sophisticated methods, the preliminary results are encouraging. Significant more engineering in the implementation of our methods and also applying our methods to instances with more general tree-decomposition structures (see for example~\cite{bergner2015automatic} for identification of sparsity structures of MIPLIB instances~\cite{gleixner2021miplib}) remain future research directions.

\newpage

\addcontentsline{toc}{chapter}{References}
\bibliographystyle{plain}
\bibliography{ref}
\newpage
\appendix{\noindent \Large\textbf{Appendix}}

In this appendix, we use $\Delta^r := \{ \xv \in \R^r_+ : \sum_{i=1}^r x_i = 1 \}$ to denote a simplex of dimension $r$.  Let $\circ:\R^n \times \R^n \to \R^n$ denote pairwise product that is for any $\xv,\yv \in \R^n, (\xv \circ \yv)_{i} := (x_i y_i) \ \forall i \in [n]$.

\section{Scenario Decomposition Algorithm of~\cite{ahmed2013scenario}}\label{sec:SDA}

\begin{algorithm}[H]
  \caption{Scenario Decomposition \cite{ahmed2013scenario} to solve (\ref{prob})}
  $\text{UB} \leftarrow \infty,\text{LB} \leftarrow -\infty,$ 
  $\mathcal{V} \leftarrow \emptyset,\lambda \leftarrow 0$
  
    \While{$\text{UB} > \text{LB}$}{
         \text{\#compute lower bound}

    \ForEach{$i \in \{1,2\}$}{
    {\scriptsize $(\xv_*^{(i)},\yv_*^{(i)}) \leftarrow \argmin\limits_{(\xv^{(i)},\yv^{(i)}) \in \mathcal{X}^{(i)}} \iprod{\cv^{(i)} + (-1)^{i} \lambda}{\xv^{(i)}} + \iprod{\dv^{(i)}}{\yv^{(i)}} : \xv^{(i)} \not\in \mathcal{V}$}
     {\scriptsize  $p_i \leftarrow \iprod{\cv^{(i)} + (-1)^{i} \lambda}{\xv^{(i)}_*} + \iprod{\dv^{(i)}}{\yv^{(i)}_*}$ }
    }

         \If{$p_1 + p_2 > \text{LB}$}{$\text{LB} \leftarrow p_1 + p_2$} 
        
    \text{\#compute upper bound}
        {\scriptsize $h_1 \leftarrow \min  \iprod{\cv^{(1)}  + \cv^{(2)}}{\xv^{(1)}_*} + \iprod{\dv^{(1)}}{\yv^{(1)}_*}   + \iprod{\dv^{(2)}}{\yv^{(2)}} : (\xv^{(1)}_*,\yv^{(2)}) \in \mathcal{X}^{(2)} $}
        {\scriptsize $h_2 \leftarrow \min  \iprod{\cv^{(1)}  + \cv^{(2)}}{\xv^{(2)}_*} + \iprod{\dv^{(1)}}{\yv^{(1)}}   + \iprod{\dv^{(2)}}{\yv^{(2)}_*} : (\xv^{(2)}_*,\yv^{(1)}) \in \mathcal{X}^{(1)} $}
         
         
         
        
        \If{$\min\{h_1,h_2\}  < \text{UB}$}{$\text{UB} \leftarrow \min\{h_1,h_2\} $} 

        
        $\mathcal{V} \leftarrow \mathcal{V} \cup \{v_1,v_2\}$

        update $\lambda$
    }
\end{algorithm}

\section{Proof of Proposition \ref{prop:redundant_constraint}}
\label{sec:redundant_constraint}
The proof is based on the primal characterization of Lagrangian relaxation \cite{geoffrion2009lagrangean}.

\begin{proof}
    It is well-known that the optimal objective function of $(\ref{ex_dual_prob})$ is equal to that of the following convex problem:
    \begin{equation}
        \begin{aligned}
            \LAG^{ex} :=  \min_{(\xv,\yv,\wv)} \ & \sum_{i \in \{1,2\}}  \iprod{\cv^{(i)}}{\xv^{(i)}} + \iprod{\dv^{(i)}}{\yv^{(i)}} \\
\text{s.t.}\ & (\xv^{(i)},\yv^{(i)},w^{(i)}) \in \conv\left\{\mathcal{X}^{(i)}_{ex}\right\},\forall i \in \{1,2\}, \\
\ & \xv^{(1)} = \xv^{(2)}, \\
\ & w^{(1)} + w^{(2)} \geq 0.
\end{aligned}
\end{equation}
Suppose that 
$F^{(1)}(\cdot) $ is some affine function, let 
\begin{equation}
    \begin{aligned}
        \mathcal{N} := \left\{ (\xv,\yv,\wv) \;\middle\vert\;
   \begin{array}{@{}l@{}} (\xv^{(i)},\yv^{(i)},w^{(i)}) \in \conv\left\{\mathcal{X}^{(i)}_{ex}\right\},\forall i \in \{1,2\}, \\ 
 \xv^{(1)} = \xv^{(2)} 
   \end{array} 
\right\}.
    \end{aligned}
\end{equation}

We prove that 
for every $(\xv,\yv,\wv) \in \mathcal{N}$, $w^{(1)} + w^{(2)} \geq 0$ is satisfied. Since $w^{(i)}$ does not contribute to the objective function and feasibility of $\xv^{(i)}$, this implies that if we remove $F^{(1)}(\xv^{(1)}) + F^{(2)}(\xv^{(2)}) \geq 0$ from $(\ref{ex_main_prob})$, then the corresponding $\LAG^{ex}$ remains the same.
To see this, consider any $(\xv,\yv,\wv) \in \mathcal{N}$, by its definition, there exists $\lambda \in \Delta^r$ such that
\begin{align*}
   & \xv^{(2)} = \sum_{j \in [r]} \lambda_j \xv^{(2)j}, \\
   & w^{(2)} = \sum_{j \in [r]} \lambda_j F^{(2)}(\xv^{(2)j}), \\
   & \xv^{(2)j} \in \mathcal{X}^{(2)},\forall j \in  [r].
\end{align*}
By (\ref{implied}), it also follows that
\begin{align*}
  F^{(1)}(\xv^{(2)j}) +  F^{(2)}(\xv^{(2)j}) \geq 0,\forall j \in [r].
\end{align*}
This implies that

\begin{subequations}
    \begin{align}
        0 \leq \sum_{j=1}^r \lambda_j F^{(1)}(\xv^{(2)j}) + \lambda_j F^{(2)}(\xv^{(2)j}) & =  F^{(1)}  (\sum_{j=1}^{r} \lambda_j \xv^{(2)j}) + w^{(2)} \label{eq_dual_redunt_1} \\
   & = F^{(1)}( \xv^{(2)}) + w^{(2)} \label{eq_dual_redunt_2} \\
   & = F^{(1)}( \xv^{(1)}) + w^{(2)} \label{eq_dual_redunt_3} \\
    & = w^{(1)}  + w^{(2)} \label{eq_dual_redunt_4} ,
    \end{align}
\end{subequations}
where (\ref{eq_dual_redunt_1}) and (\ref{eq_dual_redunt_4}) utilize the fact that $F^{(1)}(\cdot)$ is an affine function and (\ref{eq_dual_redunt_3}) utilizes the fact that $ \xv^{(1)} = \xv^{(2)} $.
\end{proof}



    

\section{Proof of Theorem~\ref{thm:A_subset_B} and Theorem~\ref{thm:strong_duality_tree}}


Theorem \ref{thm:A_subset_B} is a special case of Theorem \ref{thm:strong_duality_tree} where the underlying graph is a single edge. Therefore, proving Theorem \ref{thm:strong_duality_tree} suffices. 

The primal characterization of M-Lagrangian-dual of (\ref{tree_prob}) is
\begin{equation}
    \begin{aligned}
        \label{primal_char_ex_tree_prob}
        \min \ & \sum_{i \in V(\mathcal{T})} \iprod{\cv^{(i)}}{\xv^{(i)}} \\
        \text{s.t. } &  (\xv^{(i)},\wv^{(i)}) \in \conv\left\{\mathcal{X}^{(i)}_M\right\},\forall i \in V(\mathcal{T}), \\
        & \wv^{(i)}_S = \wv^{(j)}_S, \forall (i,j) \in E(\mathcal{T}), S \in \mathcal{S}_{ij}.
    \end{aligned}
\end{equation}

We need some technical results to study the projection of its feasible region, that is, the set  $\mathcal{A}({\mathcal{S}},\mathcal{T})$. Recall that
\begin{equation}
    \begin{aligned}
    \label{primal_ex_tree_prob_dual}
        \mathcal{A}({\mathcal{S}},\mathcal{T}) = \left\{ \xv \;\middle\vert\;
   \begin{array}{@{}l@{}} \exists \wv,
 (\xv^{(i)},\wv^{(i)}) \in \conv\left\{\mathcal{X}^{(i)}_M\right\},\forall i \in V(\mathcal{T}), \\ 
 \wv^{(i)}_S = \wv^{(j)}_S, \forall (i,j) \in E(\mathcal{T}), S \in \mathcal{S}_{ij}
   \end{array} 
\right\}.
    \end{aligned}
\end{equation}

\begin{lemma} (Simplex lemma)
    \label{lem_simplex}
    Let $\{\pv_{i} \in \mathbb{R}^n: i \in [r]\}$ be a collection of affinely independent points. For every $\xv \in \conv\left\{\pv_{i} : i \in [r]\right\}$, there exists a unique $\alpha \in \Delta^r$ such that $\xv = \sum\limits_{i \in [r]} \alpha_{i} \pv_{i}$.
\end{lemma}

\begin{proposition}
    \label{prop:proj_intersection}
    Let $\mathcal{M}_i \in \R^{s} \times \R^{t}$ be a set for all $i \in \{1,\dots,p\}$. Let $\proj_{s}\mathcal{M}_i$ denote the projection of $\mathcal{M}_i$ onto first $s$ coordinates. Then it follows that
    \begin{equation}
        \begin{aligned}
            \bigcap\limits_{i \in \{1,\dots,p\}} \proj_{s} \mathcal{M}_i \supseteq \proj_{s} \bigcap\limits_{i \in \{1,\dots,p\}}  \mathcal{M}_i.
        \end{aligned}
    \end{equation}
\end{proposition}

\begin{proposition}
    \label{prop:commute_conv_proj}
    For any arbitrary $\mathcal{M} \subseteq \R^{s} \times \R^{t}$, it follows
    $$\proj_s \conv\{\mathcal{M}\} = \conv\{\proj_s \mathcal{M}\}.$$
\end{proposition}

A direct application of Lemma \ref{lem_simplex} is the following result:
\begin{proposition}
    \label{prop:decomposibility}
    Let $\mathcal{M}_i \subseteq \R^{s} \times \R^{t_i}$  for $i \in [2]$. Suppose $\proj_s \mathcal{M}_1 \cup \proj_s \mathcal{M}_2$ is a collection of affinely independent points. For any $(\pv^{(i)}_*,\qv^{(i)}_*) \in \R^{s} \times \R^{t_i}$,
    \begin{align*}
(\pv_*^{(1)},\qv_*^{(1)},\pv_*^{(2)},\qv_*^{(2)}) \in \left\{ (\pv^{(1)},\qv^{(1)},\pv^{(2)},\qv^{(2)}) \;\middle\vert\;
   \begin{array}{@{}l@{}} 
 (\pv^{(i)},\qv^{(i)}) \in \conv\{ \mathcal{M}_i\},\forall i \in [2], \\ 
\pv^{(1)} = \pv^{(2)}  \\ 
   \end{array} 
\right\}
    \end{align*}
    if and only if 
    \begin{align*}
(\pv_*^{(1)},\qv_*^{(1)},\pv_*^{(2)},\qv_*^{(2)}) \in \conv\left\{ (\pv^{(1)},\qv^{(1)},\pv^{(2)},\qv^{(2)}) \;\middle\vert\;
   \begin{array}{@{}l@{}} 
 (\pv^{(i)},\qv^{(i)}) \in \mathcal{M}_i,\forall i \in [2], \\ 
\pv^{(1)} = \pv^{(2)}  \\ 
   \end{array} 
\right\}.
    \end{align*}
\end{proposition}

For any fixed graph $\CT$, the index list $\mathcal{S}$ completely describes the constraints of $(\ref{primal_char_ex_tree_prob})$.
Next, we observe that if we only consider the constraints induced by a certain subset of $\mathcal{S}$, we can describe the corresponding projection. 
For any $\mathcal{U} \unlhd \mathcal{S}$, because of Assumption \ref{asp:tree_down_close}, we can construct the following relaxation of the feasible region of $(\ref{primal_char_ex_tree_prob})$ induced by $\mathcal{U}$:
\begin{equation}
    \begin{aligned}
        \mathcal{H}(\mathcal{U},\cT) := \left\{ (\xv,\wv) \;\middle\vert\;
   \begin{array}{@{}l@{}} 
 (\xv^{(i)},\wv^{(i)}) \in \conv\{\mathcal{X}^{(i)}_{M}\},\forall i \in V(\cT),\\ 
 (\wv^{(i)})_S = (\wv^{(j)})_S, \forall (i,j) \in E(\mathcal{T}), S \subseteq \mathcal{U}_{ij}
   \end{array} 
\right\}.
    \end{aligned}
\end{equation}

We first observe
\begin{proposition}
    \label{prop:affine_indpt}
    For any $i \in V(\CT)$ and $(i,j) \in E(\mathcal{T})$, 
    $$\mathcal{D}^{(i)}(\mathcal{U}_{ij}) := \{\{(\wv^{(i)})_S\}_{S\subseteq \mathcal{U}_{ij}} : (\xv^{(i)}) \in \{0,1\}^{|\CQ_i|}, \wv^{(i)}_S = \prod_{j \in S} x^{(i)}_j,\forall  S\subseteq \mathcal{U}_{ij}\}$$ is a collection of affinely independent points.
\end{proposition}

\begin{proof}
For sanity check, we observe that  
every vector in $\mathcal{D}^{(i)}(\mathcal{U}_{ij})$ has dimension $2^{|\mathcal{U}_{ij}|}$ and $|\mathcal{D}^{(i)}(\mathcal{U}_{ij})| = 2^{|\mathcal{U}_{ij}|}$ \footnote{If $S = \emptyset$, we set $w^{(i)}_S = 1$.}.
Therefore, $\mathcal{D}^{(i)}(\mathcal{U}_{ij})$ satisfies the necessary condition to be a set of affinely independent points.
Moreover, since  $\textbf{0} \in \mathcal{D}^{(i)}(\mathcal{U}_{ij})$, it therefore suffices to prove that $\mathcal{D}^{(i)}(\mathcal{U}_{ij}) \setminus \{\textbf{0} \}$ is a set of linear independent points. Suppose this is not true, then there exists some $\{\lambda_l \neq 0\}$ such that
\begin{align*}
    & \sum_{l} \lambda_l \wv_l = 0, \\
    & \wv_l \in \mathcal{D}^{(i)}(\mathcal{U}_{ij}) \setminus \{\textbf{0}\}.
\end{align*}
For each $\wv_l$, we define $S_l := \{ v \in \mathcal{U}_{ij} : (\wv_l)_{\{v\}} = 1 \}$. By the construction of $\wv_l$, $(\wv_l)_S = 1$ if and only if $S \subseteq S_l$. Now choose $S^*$ to be any maximal $S_l$ among all possible choices of $l$ (such that $\lambda_l \neq 0$). We note that more than one maximal $S_l$ may exist, and we choose an arbitrary one to break a tie. Without loss of generality, we may assume that $S^* = S_1$. By construction, it follows that $(\wv_1)_{S_1} = 1$ and $(\wv_l)_{S_1} = 0$ for all $l \neq 1$ and $\lambda_l \neq 0$. This further implies that
$(\sum_{l} \lambda_l \wv_l)_{S_1} = \lambda_1 \neq 0$, which leads to contradiction.
\end{proof}




\begin{lemma}
\label{lem_single_char}
    \begin{align*}
        \proj_{\xv} \mathcal{H}(\mathcal{U},\cT) = \conv\left\{ \xv \;\middle\vert\;
   \begin{array}{@{}l@{}} 
 \xv^{(i)} \in \mathcal{X}^{(i)},\forall i \in V(\mathcal{T}), \\ 
x_v^{(i)} = x_v^{(j)},\forall v \in \mathcal{U}_{ij},\forall (i,j) \in E(\mathcal{T})  \\ 
   \end{array} 
\right\}.
    \end{align*}
\end{lemma}

\begin{proof}
Let \begin{align*}
    \mathcal{R}(\mathcal{U},\mathcal{T}) := \conv\left\{ (\xv,\wv) \;\middle\vert\;
   \begin{array}{@{}l@{}} 
 (\xv^{(i)},\wv^{(i)}) \in \mathcal{X}^{(i)}_M,\forall i \in V(\mathcal{T}), \\ 
x_v^{(i)} = x_v^{(j)},\forall v \in \mathcal{U}_{ij},\forall (i,j) \in E(\mathcal{T})  \\ 
   \end{array} 
\right\}.
\end{align*}
We first prove that $\mathcal{R}(\mathcal{U},\mathcal{T}) = \mathcal{H}(\mathcal{U},\cT)$.
Without loss of generality, we may assume that $\CT$ is connected.
We first prove that $\mathcal{R}(\mathcal{U},\mathcal{T}) \subseteq \mathcal{H}(\mathcal{U},\cT)$.
Consider any $(\xv_*,\wv_*) \in \mathcal{R}(\mathcal{U},\mathcal{T})$, by its definition, there exists $\lambda \in \Delta^r$ such that
\begin{align*}
    & (\xv_*,\wv_*) = \sum_{l \in [r]} \lambda_l (\xv_l,\wv_l), \\
    & (\xv^{(i)}_l,\wv^{(i)}_l) \in \mathcal{X}^{(i)}_M,\forall i \in V(\CT),\forall l \in [r], \\
    & (\xv_l^{(i)})_v = (\xv_l^{(j)})_v,\forall v \in \mathcal{U}_{ij},\forall (i,j) \in E(\mathcal{T}),\forall l \in [r].
\end{align*}
It is straightforward to see that $(\xv_*^{(i)},\wv^{(i)}_*) \in \conv\{\mathcal{X}^{(i)}_{M}\},\forall i \in V(\cT)$. Moreover, since $(\xv_l^{(i)})_v = (\xv_l^{(j)})_v,\forall v \in \mathcal{U}_{ij},\forall (i,j) \in E(\mathcal{T}),\forall l \in [r]$, this implies that $(\wv^{(i)}_l)_S = (\wv^{(j)}_l)_S, \forall (i,j) \in E(\mathcal{T}), S \subseteq \mathcal{U}_{ij},\forall l \in [r]$ and this further proves $(\wv^{(i)}_*)_S = (\wv^{(j)}_*)_S, \forall (i,j) \in E(\mathcal{T}), S \subseteq \mathcal{U}_{ij}$ as $\wv_*^{(i)}$ is the convex combination of $\wv_l^{(i)}$.
This conclude that $(\xv_*,\wv_*) \in \mathcal{H}(\mathcal{U},\mathcal{T})$.

To prove the reverse direction, we apply induction on $|V(\CT)|$. When $\CT$ is a single node, the statement holds trivially. Suppose the statement holds for all $\CT$ such that $|V(\CT)| \leq q$, we aim to prove the statement holds for $\CT$ such that $|V(\CT)| = q+1$. Since $\CT$ is a connected tree, there must exist a leaf in $\CT$. We denote this node by $1$ and we assume that $(1,2) \in E(\CT)$, we denote the subgraph induced by $V(\CT) \setminus \{1\}$ by $\CT'$, the corresponding index list by $\mathcal{U}' := \{\mathcal{U}'_{ij}: \mathcal{U}'_{ij} = \mathcal{U}_{ij},\forall (i,j) \in E(\CT')\}$ and the associated variable in $\mathcal{H}(\mathcal{U}',\mathcal{T}')$ by $\overline{\xv},\overline{\wv}$ and thus
we may write $\mathcal{H}(U,\mathcal{T})$ as
\begin{equation}
\label{eq:induction_hard_set_M_lag}
    \begin{aligned}
        & (\xv^{(1)},\wv^{(1)}) \in \conv\{\mathcal{X}^{(1)}_M\}, \\
    & (\overline{\xv},\overline{\wv}) \in \mathcal{H}(\mathcal{U}',\mathcal{T}'), \\
    & (\wv^{(1)})_S = (\overline{\wv}^{(2)})_S,\forall S \subseteq \mathcal{U}_{12}.
    \end{aligned}
\end{equation}
By induction hypothesis, we know that $\mathcal{R}(\mathcal{U}',\mathcal{T}') = \mathcal{H}(\mathcal{U}',\mathcal{T}')$. Therefore, we can express (\ref{eq:induction_hard_set_M_lag}) as
\begin{equation}
\label{eq:induction_hard_set_M_lag_new}
    \begin{aligned}
        & (\xv^{(1)},\wv^{(1)}) \in \conv\{\mathcal{X}^{(1)}_M\}, \\
    & (\overline{\xv},\overline{\wv}) \in \mathcal{R}(\mathcal{U}',\mathcal{T}'), \\
    & (\wv^{(1)})_S = (\overline{\wv}^{(2)})_S,\forall S \subseteq \mathcal{U}_{12}.
    \end{aligned}
\end{equation}

It is straightforward to verify that
\begin{align*}
    & \{ \{(\wv^{(1)})_S\}_{S\subseteq \mathcal{U}_{12}} : (\xv^{(1)},\wv^{(1)}) \in \mathcal{X}^{(1)}_{M}\} \subseteq \mathcal{D}^{(1)}(\mathcal{U}_{12}), \\
    & \{ \{(\wv^{(2)})_S\}_{S\subseteq \mathcal{U}_{12}} : (\xv^{(2)},\wv^{(2)}) \in \mathcal{X}^{(2)}_{M}\} \subseteq \mathcal{D}^{(2)}(\mathcal{U}_{12}), \\
    & \mathcal{D}^{(1)}(\mathcal{U}_{12}) = \mathcal{D}^{(2)}(\mathcal{U}_{12}). 
\end{align*}
By Proposition \ref{prop:affine_indpt}, $\mathcal{D}^{(1)}(\mathcal{U}_{12})$, $\mathcal{D}^{(2)}(\mathcal{U}_{12})$ are collections of affinely independent points and $\mathcal{D}^{(1)}(\mathcal{U}_{12}) \cup \mathcal{D}^{(2)}(\mathcal{U}_{12})$ is also a collection of affinely independent points since $\mathcal{D}^{(1)}(\mathcal{U}_{12}) = \mathcal{D}^{(2)}(\mathcal{U}_{12})$.

Therefore applying Proposition \ref{prop:decomposibility} to (\ref{eq:induction_hard_set_M_lag_new}), given any feasible $(\xv_*,\wv_*)$ in (\ref{eq:induction_hard_set_M_lag_new}), it follows that
\begin{align*}
    (\xv_*,\wv_*) \in 
    \conv\left\{ (\xv,\wv) \;\middle\vert\;
   \begin{array}{@{}l@{}} 
 (\xv^{(i)},\wv^{(i)}) \in \mathcal{X}^{(i)}_M,\forall i \in V(\mathcal{T}), \\ 
(\wv^{(1)})_S = (\wv^{(2)})_S,\forall S \subseteq \mathcal{U}_{12}
   \end{array} 
\right\}.
\end{align*}
where 
$\left\{ (\overline{\xv},\overline{\wv}) \;\middle\vert\;
   \begin{array}{@{}l@{}} 
 (\xv^{(i)},\wv^{(i)}) \in \mathcal{X}^{(i)}_M,\forall i \in V(\mathcal{T'}), \\ 
x_v^{(i)} = x_v^{(j)},\forall v \in \mathcal{U}_{ij},\forall (i,j) \in E(\mathcal{T'})  \\ 
   \end{array} 
\right\}$ is $\mathcal{M}_1$ 
and $\mathcal{X}_M^{(1)}$ is $\mathcal{M}_2$, $\{(\overline{\wv}^{(2)})_S\}_{S\subseteq \mathcal{U}_{12}}$ is $\pv^{(1)}$ and $\{(\wv^{(1)})_S\}_{S\subseteq \mathcal{U}_{12}}$ is $\pv^{(2)}$ in the application of Proposition \ref{prop:decomposibility}. 

Since $\{v\} \subseteq \mathcal{U}_{12},\forall v \in \mathcal{U}_{12}$, the equation $(\wv^{(1)}_l)_S = (\overline{\wv}^{(2)}_l)_S,\forall l \in [r],\forall S \subseteq \mathcal{U}_{12}$ (trivially) implies that $(\wv^{(1)}_l)_{\{v\}} = (\xv^{(1)}_l)_v = (\wv^{(2)}_l)_{\{v\}}  = (\xv^{(2)}_l)_v,\forall l \in [r],\forall v \in \mathcal{U}_{12}$. This shows that 
\begin{align*}
    (\xv_*,\wv_*) \in \conv\left\{ (\xv,\wv) \;\middle\vert\;
   \begin{array}{@{}l@{}} 
 (\xv^{(i)},\wv^{(i)}) \in \mathcal{X}^{(i)}_M,\forall i \in V(\mathcal{T}), \\ 
x_v^{(i)} = x_v^{(j)},\forall v \in \mathcal{U}_{ij},\forall (i,j) \in E(\mathcal{T'})
   \end{array} 
\right\} = \mathcal{R}(U,\mathcal{T}),
\end{align*}
and thus
finishes the induction step.

Since $\mathcal{R}(U,\mathcal{T}) = \mathcal{H}(U,\cT)$, by Proposition \ref{prop:commute_conv_proj}, it follows 
\begin{align*}
     \proj_{\xv} \mathcal{H}(U,\cT) & = \proj_{\xv}  \mathcal{R}(U,\cT) \\
     & = \proj_{\xv}  \conv\left\{ (\xv,\wv) \;\middle\vert\;
   \begin{array}{@{}l@{}} 
 (\xv^{(i)},\wv^{(i)}) \in \mathcal{X}^{(i)}_M,\forall i \in V(\CT), \\ 
x_v^{(i)} = x_v^{(j)},\forall v \in \mathcal{U}_{ij},\forall (i,j) \in E(\mathcal{T})  \\ 
   \end{array} 
\right\}  \\
& = \conv \ \proj_{\xv}  \left\{ (\xv,\wv) \;\middle\vert\;
   \begin{array}{@{}l@{}} 
 (\xv^{(i)},\wv^{(i)}) \in \mathcal{X}^{(i)}_M,\forall i \in V(\CT), \\ 
x_v^{(i)} = x_v^{(j)},\forall v \in \mathcal{U}_{ij},\forall (i,j) \in E(\mathcal{T})  \\ 
   \end{array} 
\right\} \\
& = \conv\left\{ \xv \;\middle\vert\;
   \begin{array}{@{}l@{}} 
 \xv^{(i)} \in \mathcal{X}^{(i)},\forall i \in V(\CT), \\ 
x_v^{(i)} = x_v^{(j)},\forall v \in \mathcal{U}_{ij},\forall (i,j) \in E(\mathcal{T})  \\ 
   \end{array} 
\right\}.
\end{align*}

\end{proof}
\textbf{(Proof of Theorem \ref{thm:strong_duality_tree}:)}
\begin{proof}
By Assumption \ref{asp:tree_down_close},
it follows that
\begin{align*}
    \mathcal{A}(\mathcal{S},\mathcal{T}) & = \proj_{\xv}\bigcap_{\mathcal{U} \unlhd \mathcal{S}}  \left\{ (\xv,\wv) \;\middle\vert\;
   \begin{array}{@{}l@{}} 
 (\xv^{(i)},\wv^{(i)}) \in \conv\{\mathcal{X}^{(i)}_{M}\},\forall i \in V(\CT), \\
(\wv^{(i)})_S = (\wv^{(j)})_S, \forall (i,j) \in E(\mathcal{T}), S \subseteq \mathcal{U}_{ij} \\
   \end{array}
\right\} \\
& = \proj_{\xv}\bigcap_{\mathcal{U} \unlhd \mathcal{S}} \mathcal{H}(\mathcal{U},\mathcal{T}) \\
& \subseteq \bigcap_{\mathcal{U} \unlhd \mathcal{S}} \proj_{\xv} \mathcal{H}(\mathcal{U},\mathcal{T}) \\
& = \mathcal{B}(\mathcal{S},\mathcal{T})
\end{align*}
where the containment is implied by Proposition \ref{prop:proj_intersection} and the last equality is implied by  Lemma \ref{lem_single_char}.
\end{proof}

\section{M-Lagrangian for multi-block packing and covering problems}\label{sec:M-Lagrangian for multi-block}

This section we aim to prove Theorem \ref{thm_two_block_packing_bound}, Theorem \ref{thm_two_block_covering_bound}, Theorem \ref{thm_Hierarchical_bound_for_packing} and Theorem \ref{thm_Hierarchical_bound_for_covering} and Corollary \ref{cor:application_for_two_stage}.  We comment that Theorem \ref{thm_two_block_packing_bound} and Theorem \ref{thm_two_block_covering_bound} can be regarded as special cases of Theorem \ref{thm_Hierarchical_bound_for_packing} and Theorem \ref{thm_Hierarchical_bound_for_covering} respectively,  where the underlying graph is a single edge. Thus, it suffices to prove Theorem  \ref{thm_Hierarchical_bound_for_packing} and Theorem \ref{thm_Hierarchical_bound_for_covering}.

For packing problems, we first prove the following Lemma, which relates any solution $\xv \in \mathcal{B}_k(\CT)$ to some feasible solution in (\ref{tree_prob}). 

\begin{lemma}
    \label{lem_round}
    For packing problems, under the assumption \ref{asp_local_recourse}, 
    given any good $\mathcal{W}$ and any solution $\xv \in \mathcal{B}_k(\CT)$, one can construct a feasible solution $\xv_*$ of the MIP, that is $\xv_*$ belongs to feasible region of  (\ref{tree_prob}) such that
    \begin{align*}
        \sum_{v \in \mathcal{W}} \sum_{i \in V(\CT) : v \in \CQ_i} c^{(i)}_v \cdot (\xv_*^{(i)})_v \leq \sum_{v \in \mathcal{W}} \sum_{i \in V(\CT) : v \in \CQ_i} c^{(i)}_v \cdot (\xv^{(i)})_v.
    \end{align*}
    As a consequence, since $c^{(i)}$ is non-positive we have that 
    $$\OPT \leq \sum_{v \in \mathcal{W}} \sum_{i \in V(\CT) : v \in \CQ_i} c^{(i)}_v \cdot (\xv^{(i)})_v.$$
\end{lemma}
\begin{proof}
Let $\mathcal{V} := \{ i \in V(\CT) : \CQ_i \cap  \mathcal{W} \neq \emptyset\}$. Consider the sub-graph $\CT(\mathcal{V})$ of $\CT$ induced by $\mathcal{V}$ and the collection of connected components $\{\mathcal{C}_\gamma\}_{\gamma=1}^r$ of $\CT(\mathcal{V})$. By definition of $\mathcal{W}$ being good, each $\mathcal{C}_{\gamma}$ satisfies one of the following:
\begin{enumerate}
    \item For any $(s,t) \in E(\mathcal{C}_\gamma), |\CQ_s \cap \CQ_t \cap \mathcal{W}| \leq k$. In this case, for all $j \in V(\mathcal{C}_{\gamma})$, let $\rho(j) = j$.
        \item There exists $s \in V(\mathcal{C}_\gamma)$ such that $(\CQ_j \cap \mathcal{W}) \subseteq (\CQ_s \cap \mathcal{W}),\forall j \in V(\mathcal{C}_\gamma)$. In this case, for all $j \in V(\mathcal{C}_{\gamma})$, let $\rho(j) = s$.
\end{enumerate}
If $\mathcal{C}_\gamma$ satisfies condition (1.), we call it \textit{type-1} and if $\mathcal{C}_\gamma$ satisfies condition (2.), we call it \textit{type-2}.
Since $v \in \mathcal{W}$ appears only in one connected component, we can identify the type of $v \in V(\mathcal{G}) \cap \mathcal{W}$ by saying $v$ is \textit{type-1} if $\{i : v \in \CQ_{i}\} \subseteq V(\mathcal{C}_{\gamma})$ for some type-1 $\mathcal{C}_{\gamma}$ and  $v$ is \textit{type-2} if $\{i : v \in \CQ_{i}\} \subseteq V(\mathcal{C}_{\gamma})$ for some type-2 $\mathcal{C}_{\gamma}$.

We let $\mathcal{P}_\gamma := \bigcup_{i \in V(\mathcal{C}_\gamma)} \CQ_i$, where $\mathcal{P}_\gamma$ is the collection of variables included in $\mathcal{C}_\gamma$. By the definition of tree decomposition, for each $v \in V(G)$, the set $\{i \in V(\mathcal{\mathcal{T}}) : v \in \mathcal{Q}_i \}$ forms a sub-tree of $\mathcal{T}$, this implies that 
$\mathcal{P}_{\gamma_1} \cap \mathcal{P}_{\gamma_2} = \emptyset,\forall \gamma_1 \neq \gamma_2.$ 
Moreover, by the definition of tree decomposition,  if $(u,v) \in E(\mathcal{G})$, then there exists $t \in V(\mathcal{T})$ such that $u,v \in \mathcal{Q}_t$. Therefore, by the construction of $\{\mathcal{P}_\gamma\}$, there is no constraint in (\ref{general_mip}) that involves variables from distinct pair of $\{\mathcal{P}_\gamma\}$.

Recall that 
\begin{align*}
    \mathcal{B}_k(\mathcal{T}) = \bigcap_{U \unlhd \mathcal{S}^k}  \conv \left\{ \xv \;\middle\vert\;
   \begin{array}{@{}l@{}} 
 \xv^{(i)} \in \mathcal{X}^{(i)},\forall i \in V(\mathcal{T}), \\
         x^{(i)}_v = x^{(j)}_v,\forall (i,j) \in E(\mathcal{T}),\forall v \in \mathcal{U}_{ij}.
   \end{array}
\right\}
\end{align*}
For any $\xv \in \mathcal{B}_k(\CT)$, since $k \geq 1$, it follows that $\xv^{(i)}_v = \xv^{(j)}_v,\forall v \in \CQ_{ij}$. If $v \in \mathcal{W}$ is \emph{type-2}, and $v \in \CQ_i$, this implies that  $\xv^{(i)}_v = \xv^{(\rho(i))}_v$
Also, if $v \in \mathcal{W}$ is \emph{type-1}, and $v \in \CQ_i$,
then trivially from the construction of $\rho$ we have that $\xv^{(i)}_v = \xv^{(\rho(i))}_v$.

This implies that
\begin{equation}
    \begin{aligned}
        \sum_{v \in \mathcal{W}} \sum_{i \in V(\CT) : v \in \CQ_i} c^{(i)}_v \cdot (\xv^{(i)})_v & =  \sum_{v \in \mathcal{W}} \sum_{i \in V(\CT) : v \in \CQ_i} c^{(i)}_v \cdot (\xv^{(\rho(i))})_v
    \end{aligned}
\end{equation}
For any $\xv \in \mathcal{B}_k(\CT)$, we choose $\mathcal{U}^*$ such that for each $(s,t) \in E(\CT)$,
$$\mathcal{U}_{st}^* := \begin{cases}
        \CQ_s \cap \CQ_t \cap \mathcal{W}  & \text{ if } (s,t) \in E(\mathcal{C}_\gamma) \text{ for some type-1 } \mathcal{C}_\gamma \\
     \emptyset & \text{ if otherwise}
    \end{cases}.$$
 Let $\mathcal{S}^k = \{ \mathcal{S}_{ij} : \mathcal{S}_{ij} = \binom{\CQ_{ij}}{\leq k},\forall (i,j) \in E(\CT)\}$.    
By definition of
good subsets and construction of $\mathcal{U}^*$ we have that 
$\mathcal{U}^* \unlhd \mathcal{S}^k$. Since $\xv \in \mathcal{B}_k(\CT)$ it follows that
$$\xv \in \conv \left\{ \xv \;\middle\vert\;
   \begin{array}{@{}l@{}} 
 \xv^{(i)} \in \mathcal{X}^{(i)},\forall i \in V(\mathcal{T}), \\
         x^{(i)}_v = x^{(j)}_v,\forall (i,j) \in E(\mathcal{T}),\forall v \in \mathcal{U}_{ij}^*.
   \end{array}
\right\}.$$
Therefore there exist $\lambda \in \Delta^w$ such that
\begin{align*}
     & \xv = \sum_{l \in [w]} \lambda_l \xv_l, \\
    & \xv_l^{(i)} \in \mathcal{X}^{(i)},\forall l \in [w], \\
    & (\xv_l^{(i)})_v = (\xv_l^{(j)})_v,\forall v \in \mathcal{U}_{ij}^*,\forall (i,j) \in E(\mathcal{T}).
\end{align*}

As $\xv$ is a convex combination of $\{\xv_l\}_{l=1}^w$, then there exists some $l_* \in \{1,\dots,w\}$ such that 
\begin{equation}
    \begin{aligned}
    \label{eq:convex_better}
         \sum_{v \in \mathcal{W}} \sum_{i \in V(\CT) : v \in \CQ_i} c^{(i)}_v \cdot (\xv_{l_*}^{(\rho(i))})_v & \leq \sum_{v \in \mathcal{W}} \sum_{i \in V(\CT) : v \in \CQ_i} c^{(i)}_v \cdot (\xv^{(\rho(i))})_v. \\
         & = \sum_{v \in \mathcal{W}} \sum_{i \in V(\CT) : v \in \CQ_i} c^{(i)}_v \cdot (\xv^{(i)})_v,
    \end{aligned}
\end{equation}
where again the last equality follows from the fact that $\xv \in \mathcal{B}_k(\CT)$ and since $k \geq 1$, it follows that $\xv^{(i)}_v = \xv^{(j)}_v,\forall v \in \CQ_{ij}$.

Now we choose $\xv_*$ in the following way:
\begin{align*}
    (\xv_*^{(i)})_v = \begin{cases}
        (\xv_{l_*}^{(\rho(i))})_v & \text{ if } v \in\mathcal{W} \\
        0 & \text{ if } v \not\in \mathcal{W}. 
    \end{cases}
\end{align*}
First, it is straightforward to see that 
\begin{align*}
    \sum_{v \in \mathcal{W}} \sum_{i \in V(\CT) : v \in \CQ_i} c^{(i)}_v \cdot (\xv_*^{(i)})_v  & = \sum_{v \in \mathcal{W}} \sum_{i \in V(\CT) : v \in \CQ_i} c^{(i)}_v \cdot (\xv_{l_*}^{(\rho(i))})_v  \\  
    & \leq \sum_{v \in \mathcal{W}} \sum_{i \in V(\CT) : v \in \CQ_i} c^{(i)}_v \cdot (\xv^{(i)})_v.
\end{align*}
where the last inequality comes from (\ref{eq:convex_better}).

It remains to show that $\xv_*$ is a feasible solution of (\ref{tree_prob}).

Note that $\xv_*$ satisfies all coupling constraints in (\ref{tree_prob}). For other constraints from $A\xv \leq \bv$ in (\ref{general_mip}), each such constraint is placed in some $X^{(i)}$
and it takes form of $A^{(i)} \xv^{(i)} \leq \bv^{(i)}$. To show $\xv_*$ is a feasible solution of (\ref{tree_prob}), it suffices to verify for all the constraints inside every block $\mathcal{X}^{(i)}$ is satisfied.
As discussed above, there is no constraint in (\ref{general_mip}) that involves variables from distinct pairs of $\{\mathcal{P}_\gamma\}$. So there are three cases for each $\mathcal{X}^{(i)}$ (we remind the reader that $r$ is the number of connected components):
\begin{itemize}
    \item ($\CQ_i \cap \mathcal{P}_\gamma = \emptyset,\forall \gamma \in [r]$): in this case, all variables in $\mathcal{X}^{(i)}$ are set to $0$ and therefore $ A^{(i)} \xv^{(i)} \leq \bv^{(i)}$ are satisfied since the problem type is packing.
    \item ($\CQ_i \subseteq \mathcal{P}_{\gamma^*} \text{ for some } \gamma^* \in [r]$ and $\mathcal{C}_{\gamma^*}$ is type-1): We remind that $\xv_{l_*}^{(i)} \in \mathcal{X}^{(i)}$, it therefore follows that
    \begin{align*}
        \sum_{v \in \CQ_i} A^{(i)}_v (\xv_{l_*}^{(i)})_v \leq \bv^{(i)}.
    \end{align*}
    By the way we construct $\xv_*^{(i)}$, it follows that
    \begin{align*}
            \sum_{v \in \CQ_i} A^{(i)}_v (\xv_{*}^{(i)})_v = \sum_{v \in \CQ_i \cap \mathcal{W}} A^{(i)}_v (\xv_{l_*}^{(i)})_v \leq \sum_{v \in \CQ_i} A^{(i)}_v (\xv_{l_*}^{(i)})_v \leq \bv^{(i)},
    \end{align*}
    where the last two inequalities come from the fact that the problem is packing.
    \item ($\CQ_i \subseteq \mathcal{P}_{\gamma^*}\text{ for some } \gamma^* \in [r]$ and $\mathcal{C}_{\gamma^*}$ is type-2): 
    By Assumption \ref{asp_local_recourse}, 
    it follows
    \begin{align*}
        \sum_{v \in \CQ_i \cap \CQ_{(\rho(i))}} A^{(i)}_v (\xv_{l_*}^{(\rho(i))})_v \leq \bv^{(i)}.
    \end{align*}
    Since the problem is packing, this further implies 
        \begin{align*}
        \sum_{v \in \CQ_i \cap \CQ_{(\rho(i))} \cap \mathcal{W}} A^{(i)}_v (\xv_{l_*}^{(\rho(i))})_v \leq \bv^{(i)}.
    \end{align*}
    By the way we construct $\xv_*^{(i)}$,  it follows that  
    \begin{align*}
        \sum_{v \in \CQ_i } A^{(i)}_v (\xv_{*}^{(i)})_v & =  \sum_{v \in \CQ_i \cap \mathcal{W}} A^{(i)}_v (\xv_{l_*}^{(\rho(i))})_v \\
        & =  \sum_{v \in \CQ_i \cap \CQ_{(\rho(i))} \cap \mathcal{W}} A^{(i)}_v (\xv_{l_*}^{(\rho(i))})_v \\
        & \leq \bv^{(i)}
    \end{align*}
    where the second equation comes from that fact that $\CQ_i \cap \mathcal{W} \subseteq \CQ_{(\rho(i))} \cap \mathcal{W}$.
\end{itemize}
Consequently, $\xv_*$ is a feasible solution in (\ref{tree_prob}) with the desired property.
\end{proof}

\textbf{(Proof of Theorem \ref{thm_Hierarchical_bound_for_packing}):}
\begin{proof}
Give any good $\mathcal{W}$ and any $i \in V(\CT)$, we define
\begin{align*}
   \chi_\mathcal{W}^{(i)} \in \{0,1\}^{\CQ_i} : (\chi_\mathcal{W}^{(i)})_v = \begin{cases}
        1 & \text{ if } v \in \CQ_i \cap \mathcal{W} \\
        0 & \text{ if } v \in \CQ_i \setminus \mathcal{W}.
    \end{cases}
\end{align*}
We also use $\ev^{(i)}$ to denote a vector of all ones with dimension $|\CQ_i|$.

Given any solution in (\ref{multi_block_packing_bound}), it is straightforward to see that
\begin{align*}
    \sum_{\mathcal{W} : \mathcal{W} \text{ is good} } \alpha_{\mathcal{W}} \chi_{\mathcal{W}} \geq \textbf{1} \implies  \sum_{\mathcal{W} : \mathcal{W} \text{ is good} } \alpha_{\mathcal{W}} \chi_{\mathcal{W}}^{(i)} \geq \ev^{(i)},\forall i \in V(\CT).
\end{align*}
Let $\circ$ denote pairwise product. 
Now consider any optimal solution $\xv$ of (\ref{primal_char_ex_tree_prob}) and let $\alpha^*$ denote the optimal solution of (\ref{multi_block_packing_bound}),
it follows that
\begin{align*}
    \LAG^{M}_k & = \sum_{i \in V(\CT)} \iprod{\cv^{(i)}}{\xv^{(i)}} \\
    & = \sum_{i \in V(\CT)} \iprod{\cv^{(i)}}{\ev^{(i)} \circ \xv^{(i)}} \\
    & \geq  \sum_{i \in V(\CT)} \iprod{\cv^{(i)}}{\left(\sum_{\mathcal{W} : \mathcal{W} \text{ is good} } \alpha_{\mathcal{W}}^* \chi_{\mathcal{W}}^{(i)}\right) \circ \xv^{(i)}} \\
    & =  \sum_{\mathcal{W}: \mathcal{W} \text{ is good}} \alpha_{\mathcal{W}}^* \sum_{i \in V(\CT)} \iprod{\cv^{(i)}}{\chi_{\mathcal{W}}^{(i)} \circ \xv^{(i)}} \\
    & \geq \left(\sum_{\mathcal{W}: \mathcal{W} \text{ is good}} \alpha_{\mathcal{W}}^*\right) \cdot \OPT \\
    & = \eta_k \cdot \OPT
\end{align*}
where the first inequality comes from the fact that the problem is packing, $\cv^{(i)}$ is non-positive and $\sum_{\mathcal{W} : \mathcal{W} \text{ is good} } \alpha_{\mathcal{W}}^* \chi_{\mathcal{W}}^{(i)} \geq \ev^{(i)},\forall i \in V(\CT)$; the second inequality is the consequence of Lemma \ref{lem_round} and $\alpha^* \geq 0$.
\end{proof}

\textbf{(Proof of Theorem \ref{thm_Hierarchical_bound_for_covering}):}
\begin{proof}
Give any $k$-good $\mathcal{W}$ and any $i \in V(\mathcal{T})$, we let
\begin{align*}
   \chi_\mathcal{W}^{(i)} \in \{0,1\}^{\CQ_i} : (\chi_\mathcal{W}^{(i)})_v = \begin{cases}
        1 & \text{ if } v \in \CQ_i \cap \mathcal{W} \\
        0 & \text{ if } v \in \CQ_i \setminus \mathcal{W}.
    \end{cases}
\end{align*}
We also use $\ev^{(i)}$ to denote a vector of all ones with dimension $|\CQ_i|$.

Given any solution in (\ref{multi_block_covering_bound}), it is straightforward to see that
\begin{align*}
    \sum_{\mathcal{W} : \mathcal{W} \text{ is $k$-good} } \alpha_{\mathcal{W}} \chi_{\mathcal{W}} \geq \textbf{1} \implies  \sum_{\mathcal{W} : \mathcal{W} \text{ is $k$-good} } \alpha_{\mathcal{W}} \chi_{\mathcal{W}}^{(i)} \geq \ev^{(i)},\forall i \in V(\CT).
\end{align*}

Now consider any optimal solution $\xv_*$ in (\ref{primal_char_ex_tree_prob}) and let $\alpha^*$ denote the optimal solution in (\ref{multi_block_covering_bound}), it follows that
\begin{align*}
    \LAG_k^M & = \sum_{i \in V(\CT)} \iprod{\cv^{(i)}}{\xv^{(i)}_*} \\
    & = \sum_{i \in V(\CT)} \iprod{\cv^{(i)}}{\ev^{(i)} \circ \xv^{(i)}_*} \\
    & \leq  \sum_{i \in V(\CT)} \iprod{\cv^{(i)}}{\left(\sum_{\mathcal{W} : \mathcal{W} \text{ is $k$-good} } \alpha_{\mathcal{W}}^* \chi_{\mathcal{W}}^{(i)}\right) \circ \xv^{(i)}_*} \\
    & =  \sum_{\mathcal{W}: \mathcal{W} \text{ is $k$-good}} \alpha_{\mathcal{W}}^* \sum_{i \in V(\CT)} \iprod{\cv^{(i)}}{\chi_{\mathcal{W}}^{(i)} \circ \xv^{(i)}_*} \\
\implies & \dfrac{\LAG_k^M}{\theta_k} \leq \sum_{\mathcal{W}: \mathcal{W} \text{ is $k$-good}} \dfrac{\alpha_{\mathcal{W}}^*}{\theta_k} \sum_{i \in V(\CT)} \iprod{\cv^{(i)}}{\chi_{\mathcal{W}}^{(i)} \circ \xv^{(i)}_*}
\end{align*}

Since $\dfrac{\alpha_{\mathcal{W}}^*}{\theta_k} \geq 0 \text{ for all $k$-good $\mathcal{W}$ and } \sum_{\mathcal{W}: \mathcal{W} \text{ is $k$-good}} \dfrac{\alpha_{\mathcal{W}}^*}{\theta_k} = 1
$, then it follows that there exists a $k$-good $\mathcal{W}^*$ such that $\sum_{i \in V(\CT)} \iprod{\cv^{(i)}}{\chi_{\mathcal{W^*}}^{(i)} \circ \xv^{(i)}_*} \geq \dfrac{\LAG_k^M}{\theta_k}$.

Let $\mathcal{V} := \{ i \in V(\CT) : \CQ_i \cap  \mathcal{W}^* \neq \emptyset\}$. Consider the subgraph $\CT(\mathcal{V})$ of $\CT$ induced by $\mathcal{V}$ and the collection of connected components $\{\mathcal{C}_\gamma\}_{\gamma=1}^r$ of $\CT(\mathcal{V})$.  
We first choose $\mathcal{U}^*$ such that for each $(s,t) \in E(\CT)$,
$$\mathcal{U}_{st}^* := \begin{cases}
        \CQ_s \cap \CQ_t \cap \mathcal{W}^*  & \text{ if } (s,t) \in E(\mathcal{C}_\gamma) \\
     \emptyset & \text{ if otherwise}
    \end{cases}.$$
Let $\mathcal{S}^k = \{ \mathcal{S}_{ij} : \mathcal{S}_{ij} = \binom{\CQ_{ij}}{\leq k},\forall (i,j) \in E(\CT)\}$.    
Recall that 
\begin{align*}
    \mathcal{B}_k(\mathcal{T}) = \bigcap_{U \unlhd \mathcal{S}^k}  \conv \left\{ \xv \;\middle\vert\;
   \begin{array}{@{}l@{}} 
 \xv^{(i)} \in \mathcal{X}^{(i)},\forall i \in V(\mathcal{T}), \\
         x^{(i)}_v = x^{(j)}_v,\forall (i,j) \in E(\mathcal{T}),\forall v \in \mathcal{U}_{ij}.
   \end{array}
\right\}.
\end{align*}
Since $\xv_* \in \mathcal{A}_k(\CT) \subseteq \mathcal{B}_k(\CT)$ and $\mathcal{U}^* \unlhd \mathcal{S}^k$ as $\mathcal{W}^* \text{ is $k$-good}$,
it, therefore, follows that
$$\xv_* \in \conv \left\{ \xv \;\middle\vert\;
   \begin{array}{@{}l@{}} 
 \xv^{(i)} \in \mathcal{X}^{(i)},\forall i \in V(\mathcal{T}), \\
         x^{(i)}_v = x^{(j)}_v,\forall (i,j) \in E(\mathcal{T}),\forall v \in \mathcal{U}_{ij}^*.
   \end{array}
\right\}.$$
Moreover since $\xv_*$ is feasible in (\ref{primal_char_ex_tree_prob}), then it follows that
\begin{equation}
    \begin{aligned}
        \label{eq:tree_all_same}
        (\xv_*^{(i)})_v = (\xv_*^{(j)})_v,\forall i,j \in V(\CT) : v \in \CQ_{ij}.
    \end{aligned}
\end{equation}

By definition, there exist $\lambda \in \Delta^w$ such that
\begin{equation}
    \begin{aligned}
    \label{eq:dfn_xv*}
    & \xv_* = \sum_{l \in [w]} \lambda_l \xv_l, \\
    & \xv_l^{(i)} \in \mathcal{X}^{(i)},\forall l \in [w], \\
    & (\xv_l^{(i)})_v = (\xv_l^{(j)})_v,\forall v \in \mathcal{U}_{ij}^*,\forall (i,j) \in E(\mathcal{T}).
    \end{aligned}
\end{equation}



Now we construct $\overline{\xv}_l$ in the following way:
\begin{align*}
    (\overline{\xv}_l)_v = \begin{cases}
        (\xv_{l}^{(i)})_v & \text{ if } v \in \mathcal{W}^*, \\
        \max_{i : v \in \CQ_i}\{(\xv_{l}^{(i)})_v\} & \text{ if otherwise} 
    \end{cases}
\end{align*} 
and $$\overline{\xv} = \sum \lambda_l \overline{\xv}_l.$$
It is therefore straightforward that $\overline{\xv}_l$ satisfies all coupling constraints in $(\ref{tree_prob})$ and all other constraints absorbed in $\mathcal{X}^{(i)}$ since the problem is covering type and each entry of $\overline{\xv}_l$ is at least as large as that of $\xv_l$. Therefore, $\overline{\xv}_l$ is feasible in $(\ref{tree_prob})$ and $\sum_{i \in V(\CT)} \iprod{\cv^{(i)}}{\overline{\xv}^{(i)}} \geq \OPT$. Now let us compare $\xv_*$ to $\overline{\xv}$. We first observe that
\begin{equation}
    \begin{aligned}
        \label{eq:max_bound}
        \max_{i : v \in \CQ_i}\{(\xv_{l}^{(i)})_v\} \leq \sum_{i : v \in \CQ_i} (\xv_{l}^{(i)})_v \text{ since } (\xv_{l}^{(i)})_v \geq 0.
    \end{aligned}
\end{equation}
By construction, for any $v \in \mathcal{W}^*$, it follows that
\begin{align*}
    (\overline{\xv}^{(i)})_v = (\xv_*^{(i)})_v,\forall v \in \mathcal{W}^*,\forall i \in V(\CT)
\end{align*}
and for any other $v \notin \mathcal{W}^*$ and any $i \in V(\CT)$ such that $v \in \CQ^{(i)}$, it follows that
\begin{align*}
(\overline{\xv}^{(i)})_v & = \sum_{l=1}^w \lambda_l (\overline{\xv}_l^{(i)})_v \\
& \leq \sum_{l=1}^w \lambda_l \sum_{j : v \in \CQ_j} (\xv_{l}^{(j)})_v \\
& = \sum_{j : v \in \CQ_j} \sum_{l=1}^w \lambda_l  (\xv_{l}^{(j)})_v \\
& = \sum_{j :v \in \CQ_i} (\xv_*^{(j)})_v \\
& \leq \tau \cdot ({\xv_*}^{(i)})_v,
\end{align*}
where the first inequality comes from (\ref{eq:max_bound}); the third equation comes from (\ref{eq:dfn_xv*}); the second inequality comes from (\ref{eq:tree_all_same}) and the choice of $\tau$.

This yields that
\begin{align*}
    \OPT & \leq \sum_{i \in V(\CT)} \iprod{\cv^{(i)}}{\overline{\xv}^{(i)}} \\
    & = \sum_{i \in V(\CT)} \iprod{\cv^{(i)}}{\chi_{\mathcal{W^*}}^{(i)} \circ \overline{\xv}^{(i)}} + \sum_{i \in V(\CT)} \iprod{\cv^{(i)}}{(\textbf{1} - \chi_{\mathcal{W^*}}^{(i)}) \circ\overline{\xv}^{(i)}} \\
    & \leq \sum_{i \in V(\CT)} \iprod{\cv^{(i)}}{\chi_{\mathcal{W^*}}^{(i)} \circ \xv_*^{(i)}} + \tau \sum_{i \in V(\CT)} \iprod{\cv^{(i)}}{(\textbf{1} - \chi_{\mathcal{W^*}}^{(i)}) \circ \xv_*^{(i)}} \\
\end{align*}

Recall that \begin{align*}
    & \sum_{i \in V(\CT)} \iprod{\cv^{(i)}}{\chi_{\mathcal{W^*}}^{(i)} \circ \xv_*^{(i)}} \geq \dfrac{\LAG_k^M}{\theta_k}, \\
    & \sum_{i \in V(\CT)} \iprod{\cv^{(i)}}{(\textbf{1} - \chi_{\mathcal{W^*}}^{(i)}) \circ \xv_*^{(i)}} = \LAG_k^M - \sum_{i \in V(\CT)} \iprod{\cv^{(i)}}{\chi_{\mathcal{W^*}}^{(i)} \circ \xv_*^{(i)}}.
\end{align*}

It follows that
\begin{align*}
    \dfrac{\OPT}{\LAG_k^M} & \leq \dfrac{\sum_{i \in V(\CT)} \iprod{\cv^{(i)}}{\chi_{\mathcal{W^*}}^{(i)} \circ \xv_*^{(i)}} + \tau \sum_{i \in V(\CT)} \iprod{\cv^{(i)}}{(\textbf{1} - \chi_{\mathcal{W^*}}^{(i)}) \circ \xv_*^{(i)}}}{\LAG_k^M} \\
    & =  \dfrac{\sum_{i \in V(\CT)} \iprod{\cv^{(i)}}{\chi_{\mathcal{W^*}}^{(i)} \circ \xv_*^{(i)}} + \tau (\LAG_k^M - \sum_{i \in V(\CT)} \iprod{\cv^{(i)}}
    {\chi_{\mathcal{W^*}}^{(i)} \circ \xv_*^{(i)}})}{\LAG_k^M} \\
     & =  \dfrac{(1-\tau) \sum_{i \in V(\CT)} \iprod{\cv^{(i)}}{\chi_{\mathcal{W^*}}^{(i)} \circ \xv_*^{(i)}} + \tau \LAG_k^M}{\LAG_k^M} \\
     & \leq \dfrac{(1-\tau) \dfrac{\LAG_k^M}{\theta_k} + \tau \LAG_k^M}{\LAG_k^M} \\
    & = \dfrac{(1-\tau) + \tau \cdot\theta_k}{\theta_k}.
\end{align*}
Therefore, we conclude that
\begin{align*}
   \dfrac{\theta_k}{(1-\tau) + \tau \cdot\theta_k} \OPT \leq  \LAG_k^M.
\end{align*} 
\end{proof}

(\textbf{Proof of Corollary \ref{cor:application_for_two_stage}}):
\begin{proof}
    Consider the following two-stage problem with $Z$ scenarios:
    \begin{equation}
        \begin{aligned}
            \label{eq:two_stage_problem}
            \OPT :=  \min_{(\xv,\yv)} \ & \sum_{i \in [k]} \iprod{\cv^{(i)}}{\xv^{(i)}} + \iprod{\dv^{(i)}}{\yv^{(i)}} \\
\text{s.t.} \ & (\xv^{(i)},\yv^{(i)}) \in \mathcal{X}^{(i)},\forall i \in [Z], \\
\ & \xv^{(1)} = \xv^{(j)} \in \{0,1\}^n,\forall j \in [Z].
        \end{aligned}
\end{equation}
In this case, we can view $(\ref{eq:two_stage_problem})$ as a MIP with variables $\xv,\yv^{(i)},\forall i \in [Z]$ where $\xv$ has a local copy $\xv^{(i)}$ in each block $i$.
When each block is a packing problem, it is straightforward to verify that that 
\begin{align*}
   & \mathcal{F}_i := \{\xv\} \cup \{\yv^{(i)}\},\forall i \in [Z], \\
    & \mathcal{W}_S := \{\xv_S\} \cup \bigcup_{i \in [Z]} \{\yv^{(i)}\},\forall S \in \binom{n}{\leq k}.
\end{align*}
are all good. Therefore, to apply Theorem \ref{thm_Hierarchical_bound_for_packing}, we consider the following LP: 
\begin{align*}
    \eta_k := \min & \sum_{i \in [Z]} \alpha_i + \sum_{S \in \binom{n}{\leq k}}  \beta_{S} \\  \text{s.t. } &\sum_{i \in [Z]} \alpha_i \chi_{\mathcal{F}_i} + \sum_{S \in \binom{n}{\leq k}}  \beta_s \chi_{\mathcal{W}_S} \geq \textbf{1} \\ & \alpha,\beta \geq 0.
\end{align*}
Due to the symmetry, one may assume that $\alpha_i$ are the same for $i \in [K]$ and $\beta_S$ are the same for $S \in \binom{n}{\leq k}$. Therefore, the above LP reduces to the following two-dimensional LP:
\begin{align*}
    \eta_k := \min & \ \alpha + \beta \\  \text{s.t. } & 
    \begin{bmatrix}
        1 \\ 1/Z
    \end{bmatrix} \alpha + \begin{bmatrix}
        t \\ 1
    \end{bmatrix} \beta \geq \textbf{1} \\ & \alpha,\beta \geq 0.
\end{align*}
and one can verify that $\eta_k = 2+\dfrac{2t-Z\cdot t-1}{Z-t}$ and therefore under Assumption \ref{asp_local_recourse}, $\left(2+\dfrac{2t-Z\cdot t-1}{Z-t}\right) \cdot \OPT \leq \LAG_k^{M} \leq \OPT$.

When each block is a covering problem, it is straightforward to verify that that 
\begin{align*}
    & \mathcal{W}_S := \{\xv_S\} \cup \bigcup_{i \in [Z]} \{\yv^{(i)}\},\forall S \in \binom{n}{\leq k}.
\end{align*}
are all $k$-good.
Similar to the argument above, to apply the Theorem \ref{thm_Hierarchical_bound_for_covering}, we can consider the following LP (reduced by symmetry):
\begin{align*}
    \theta_k := \min & \ \beta \\  \text{s.t. } & 
     \begin{bmatrix}
        t \\ 1
    \end{bmatrix} \beta \geq \textbf{1} \\ & \beta \geq 0.
\end{align*}
and one can verify that $\theta_k = \dfrac{1}{t}, \tau = Z$. Therefore, $\dfrac{1}{(1-Z)\cdot t + Z} \cdot \OPT \leq \LAG_k^{M} \leq \OPT$.
\end{proof}

\textbf{(Proof of Remark \ref{rem:tightness}):}

\begin{proof}
For packing instance, we consider
\begin{align*}
C^{(1)}(x_1,x_2) = &  \min -0.5 y_1 -  0.25 x_1 - 0.25 x_2 \\
& \text{s.t. } y_1 + x_1 + x_2 \leq 2, \\
& \ \ \ \ \ y_1,x_1,x_2 \in \{0,1\}. 
\end{align*}
and
\begin{align*}
C^{(2)}(x_3,x_4) = &  \min -0.5 y_2 -  0.25 x_3 - 0.25 x_4 \\
& \text{s.t. } 2y_2 + x_3 + x_4 \leq 2, \\
& \ \ \ \ \ y_2,x_3,x_4 \in \{0,1\}. 
\end{align*}
It is easy to see that
\begin{align*}
    & C^{(1)}(0,0) = -0.5,  C^{(1)}(0,1) = -0.75,  C^{(1)}(1,0) = -0.75, C^{(1)}(1,1) = -0.5, \\
    & C^{(2)}(0,0) = -0.5,  C^{(2)}(0,1) = -0.25,  C^{(2)}(1,0) = -0.25, C^{(2)}(1,1) = -0.5.
\end{align*}

In this case, it follow that $\OPT = -1$. It is also straightforward to check that $(x_1,x_2,y_1,x_3,x_4,y_2)  = (0.5,0.5,1,0.5,0.5,0.5)$ is feasible in (\ref{primal_char_ex_tree_prob}) and therefore $\LAG^M_1 \leq -1.25$ and therefore $\dfrac{\LAG^M_1}{\OPT} \geq 1.25$.

For covering instance, we consider 
\begin{align*}
C^{(1)}(x_1,x_2) = &  \min  0.25 x_1 + 0.25 x_2 \\
& \text{s.t. } y_1 + x_1 + x_2 \geq 2, \\
& \ \ \ \ \ y_1,x_1,x_2 \in \{0,1\}. 
\end{align*}
and
\begin{align*}
C^{(2)}(x_3,x_4) = &  \min 0.5 y_2 +  0.25 x_3 +0.25 x_4 \\
& \text{s.t. } 2y_2 + x_3 + x_4 \geq 2, \\
& \ \ \ \ \ y_2,x_3,x_4 \in \{0,1\}. 
\end{align*}
It is easy to see that
\begin{align*}
    & C^{(1)}(0,0) = \infty,  C^{(1)}(0,1) = 0.25,  C^{(1)}(1,0) = 0.25, C^{(1)}(1,1) = 0.5, \\
    & C^{(2)}(0,0) = 0.5,  C^{(2)}(0,1) = 0.75,  C^{(2)}(1,0) = 0.75, C^{(2)}(1,1) = 0.5. 
\end{align*}
In this case, it follow that $\OPT = 1$. It is also straightforward to check that $(x_1,x_2,y_1,x_3,x_4,y_2)  = (0.5,0.5,1,0.5,0.5,0.5)$ is feasible in (\ref{primal_char_ex_tree_prob}) and therefore $\LAG^M_1 \leq 0.75$ and therefore $\dfrac{\LAG^M_1}{\OPT} \leq 0.75$.

\end{proof}

\section{Strong duality of V-Lagrangian}
\label{sec:Strong duality of V-Lagrangian}
In this section, we focus on the proof of Theorem \ref{thm_strong_dual_vertex}.

\textbf{(Proof of Theorem \ref{thm_strong_dual_vertex}:)}
\begin{proof}
    As we mentioned earlier, the optimal objective value of $(\ref{ex_dual_prob_vertex})$ is equal to the optimal objective function value of the following convex problem:
\begin{equation}
\label{ex_main_prob_vertex_primal}
\begin{aligned}
\LAG^V :=  \min_{(\xv,\yv,\wv)} \ & \sum_{i \in \{1,2\}}  \iprod{\cv^{(i)}}{\xv^{(i)}} + \iprod{\dv^{(i)}}{\yv^{(i)}} \\
\text{s.t.}\ & (\xv^{(i)},\yv^{(i)},\wv^{(i)}) \in \conv\{\mathcal{X}^{(i)}_{V}\},\forall i \in \{1,2\}, \\
& \  \xv^{(1)} = \xv^{(2)}, \\
& \ \wv^{(1)} = \wv^{(2)}.
\end{aligned}
\end{equation}


Consider $$\mathcal{J}^{(i)}  := \{(\xv^{(i)},\wv^{(i)}) : \xv^{(i)} \in \{0,1\}^n, w^{(i)}_{\vv} = \prod_{s \in [n]} \sigma_{v_s}((\xv^{(i)})_s),\text{$\forall \vv$ vertex of $[0,1]^n$}\}, \forall i \in \{1,2\}.$$
It is straightforward to verify that
\begin{align*}
& \mathcal{J}^{(i)} \text{ is a collection of affinely independent points}, \forall i \in \{1,2\}, \\
& \mathcal{J}^{(1)} = \mathcal{J}^{(2)}.
\end{align*}
By Proposition \ref{prop:decomposibility},
the feasible region of (\ref{ex_main_prob_vertex_primal}) is the same as
$$\conv\left\{ (\xv,\yv,\wv) \;\middle\vert\;
   \begin{array}{@{}l@{}} 
 (\xv^{(i)},\yv^{(i)},\wv^{(i)}) \in \mathcal{X}^{(i)}_V,\forall i \in \{1,2\}, \\ 
 \xv^{(1)} = \xv^{(2)}, \\
\wv^{(1)} = \wv^{(2)}\\ 
   \end{array} 
\right\}. $$
where $\mathcal{X}^{(i)}_{V}$ is $\mathcal{M}_i$ and $ (\xv^{(i)},\wv^{(i)}) = \pv_i$ in the application of Proposition \ref{prop:decomposibility}.

Therefore, for any feasible solution $(\xv_*,\yv_*,\wv_*)$ in (\ref{ex_main_prob_vertex_primal}), it follows that $(\xv_*,\yv_*)$ is a convex combination of feasible solution in (\ref{prob}). Since the objective function is a linear function and only involves $\xv,\yv$, this implies $\OPT = \LAG^V$.
\end{proof}

\section{Comparison between V-Lagrangian and SDA}\label{sec:Comparison between V-Lagrangian and SDA}
In this section, we will present proofs for Proposition \ref{prop_comparision} and Proposition \ref{prop:SDA_bad}.

In the following subsection, we collect some standard results from convex optimization that we require for our results. 
\subsection{Analysis of sub-gradient method.}

\begin{lemma}
Since $L^{V}(\lambda,\mu)$ is concave piece-wise linear function, its $\partial L^{\mathcal{V}}(\lambda,\mu)$ is defined as
\begin{align*}
    \partial L^{V}(\lambda,\mu) = \conv\{\hv : \hv \in \mathcal{I}(\lambda,\mu)\}
\end{align*}
where $$\mathcal{I}(\lambda,\mu) := \{ (\xv^{(1)},\wv^{(1)}) - (\xv^{(2)},\wv^{(2)}) : (\xv,\wv) \text{ is minimizer of } L^{V}(\lambda,\mu) \}.$$    
\end{lemma}

\begin{corollary}
For any $\nabla L^{V}(\lambda,\mu) \in \partial L^{V}(\lambda,\mu)$, $\norm{\nabla L^{V}(\lambda,\mu)}_2^2 \leq n+2$.
\end{corollary}
\begin{proof}
    For any optimal solution of one of the sub-problems of  $L^{V}(\lambda,\mu)$, every component of $(\xv,\wv)$ is in $\{0, 1\}$ and the support of $w$ is one. 
\end{proof}

We need the following classic result of non-smooth optimization that we prove here for completeness:
\begin{proposition}
    \label{prop:gradient}
    Let $(\lambda^*,\mu^*)$ be optimal solution of (\ref{ex_dual_prob_vertex}). Then there is a sub-gradient method that solves (\ref{ex_dual_prob_vertex}) to $\epsilon$-optimality in $\dfrac{\norm{\lambda^*,\mu^*}^2(n+2)}{\epsilon^2}$ iterations. 
\end{proposition}

\begin{proof}
    We present a general result on the sub-gradient method for non-smooth convex optimization. Consider some arbitrary unconstrained convex function $f(\xv)$ \footnote{The "$x$" variable here is a generic variable, not to be confused with the $x$ variables in the MIPs we consider in the paper.}, its sub-gradient $\partial f(\xv)$ is defined as
\begin{align*}
    \partial f(\xv) := \{ \gv : f(\yv) \geq f(\xv) + \iprod{\gv }{\yv - \xv},\forall \yv\}.
\end{align*}

The classic sub-gradient algorithm is as follows:
\begin{algorithm}
\caption{classic sub-gradient method}\label{alg:two}
 initialize $\xv_1 \text{ and } t = 1$ \\
 \While{some termination condition is not met}{
  \text{compute } $\gv_t \in \partial f(\xv_t)$  \\
  $\xv_{t+1} = \xv_{t} - \mu_t \gv_t$ \Comment{$\mu_t$ is the step size}
 }
\end{algorithm}

Let $G$ be the Lipschitz constant of $f$ such that
\begin{align*}
    \norm{f(\xv) - f(\yv)} \leq G \norm{\xv - \yv}_2,\forall \xv,\yv
\end{align*}
which is equivalent to
\begin{align*}
    \norm{ \gv(\xv)} \leq G,\forall \xv,\forall \gv(\xv) \in \partial f(\xv).
\end{align*}

Consider any optimal solution $\xv_*$ and let $R := \norm{\xv_1-\xv_*}_2$. Let $\{\mu_r\}_{r=1}^t$ be a sequence of step-size.
Using the definition of sub-gradient, it follows that
\begin{align}
    \norm{\xv_{t+1} - \xv_*}_2^2 & =  \norm{\xv_{t} - \mu_{t} \gv_{t} - \xv_*}_2^2 \\
    & = \norm{\xv_{t}  - \xv_*}_2^2 + \mu_{t}^2 \norm{\gv_{t}}_2^2 - 2 \mu_{t}  \iprod{\gv_{t}}{\xv_{t} - \xv_*}   \\
    & \leq \norm{\xv_{t}  - \xv_*}_2^2 + \mu_{t}^2 \norm{\gv_{t}}_2^2 - 2 \mu_{t}  (f(\xv_{t}) - f(\xv_*)) \label{eq:telescope}
\end{align}

The telescoping  summation of (\ref{eq:telescope}), implies that
\begin{align}
    \label{eq:subgradient1}
    \norm{\xv_{t+1} - \xv_*}_2^2 \leq \norm{\xv_{1} - \xv_*}_2^2 - 2\sum_{r=1}^{t} \mu_r (f(\xv_r) - f(\xv_*)) + \sum_{r=1}^t \mu_{r}^2 \norm{\gv_r}_2^2
\end{align}
Since $\norm{\xv_{t+1} - \xv_*}_2^2 \geq 0$, (\ref{eq:subgradient1}) implies that
\begin{equation}
\label{eq:subgradient2}
    \begin{aligned}
        & 0 \leq R^2 - 2\sum_{r=1}^{t} \mu_r (f(\xv_r) - f(\xv_*)) + G^2 \sum_{r=1}^t \mu_{r}^2.  \\
    \implies  & 2\sum_{r=1}^{t} \mu_r (f(\xv_r) - f(\xv_*))  \leq R^2 + G^2 \sum_{r=1}^t \mu_{r}^2.
    \end{aligned}
\end{equation}
Let $f^{\text{best}}_t = \min\limits_{r=1,\dots,t} f(\xv_r) - f(\xv_*)$. Then by (\ref{eq:subgradient2}), it follows that
\begin{align}
    f^{\text{best}}_t \leq \dfrac{R^2 + G^2 \sum\limits_{r=1}^t \mu_r^2 }{2\sum\limits_{r=1}^t \mu_{r}}. \label{eq:subgradient3}
\end{align}
If we use fixed step size $\mu_r = \mu$, (\ref{eq:subgradient3}) becomes
\begin{align}
    f^{\text{best}}_t \leq \dfrac{R^2}{2t \mu} + \dfrac{G^2 \mu}{2}.
\end{align}
Choosing $\mu = \epsilon/G^2$ and $t=R^2G^2/\epsilon^2$, it follows that
\begin{align*}
    f^{\text{best}}_t \leq \epsilon
\end{align*}
This gives a convergence proof with $R^2G^2/\epsilon^2$ iterations. 

The result now follows from the fact that in our setting, $G \leq \sqrt{n+2}.$

\end{proof}

\subsection{Proof of Proposition \ref{prop_comparision}}
Since the dimension of $(\lambda^*,\mu^*)$ is exponential in $n$, it is possible that $\norm{\lambda^*,\mu^*}$ is very large.  We divide the proof of Proposition \ref{prop_comparision} into two cases, which bounds the size of $\norm{\lambda^*,\mu^*}$. These two cases are presented as Proposition $\ref{prop:terminate1}$ and Proposition $\ref{prop:terminate2}$.

\begin{proposition}
        \label{prop:terminate1} Suppose $\cv^{(i)},\dv^{(i)},\xv^{(i)},\yv^{(i)} \geq 0$ for all $i \in \{1,2\}$.
        If SDA terminates in $t$ iterations while the sub-problems remain feasible in all iterations of the SDA algorithm, then there exists some optimal solution $(\lambda^*,\mu^*)$ of (\ref{ex_dual_prob_vertex}) such that $\norm{(\lambda^*,\mu^*)}^2 \leq 8 \phi^2 t$.
\end{proposition}

\begin{proof}
    Suppose the SDA terminates in $t$ iterations; we denote the final $\mathcal{V}$ by $\mathcal{V}_t$. For any $\vv \in \{0,1\}^n$, we define
    \begin{align*}
        & p^{(1)}(\vv) := \min_{\yv^{(1)}} \iprod{\cv^{(1)}}{\vv} + \iprod{\dv^{(1)}}{\yv^{(1)}} \text{ s.t. } (\vv,\yv^{(1)}) \in \mathcal{X}^{(1)}, \\
        & p^{(2)}(\vv) := \min_{\yv^{(2)}} \iprod{\cv^{(2)}}{\vv} + \iprod{\dv^{(2)}}{\yv^{(2)}} \text{ s.t. } (\vv,\yv^{(2)}) \in \mathcal{X}^{(2)}, \\
        & p^{(1)}_{low} :=  \min p^{(1)}(\uv)  \text{ s.t. } \uv \in \{0,1\}^n \setminus \mathcal{V}_t, \\
        & p^{(2)}_{low} :=  \min p^{(2)}(\uv)  \text{ s.t. } \uv \in \{0,1\}^n \setminus \mathcal{V}_t.
    \end{align*}
Since $\cv^{(i)},\dv^{(i)},\xv^{(i)},\yv^{(i)}$ are non-negative, it follows that $p^{(i)}(\vv) \geq 0,\forall \vv \in \{0,1\}^n,\forall i \in \{1,2\}$ and $p^{(i)}_{low} \geq 0,\forall i \in \{1,2\}$.

Let the optimal solution $\vv^* \in \mathcal{V}_t$, by termination condition and optimality condition, we obtain
    \begin{equation}
        \label{two_block_vertex_opt_condition}
        \begin{aligned}
                   & p^{(1)}(\vv^*) + p^{(2)}(\vv^*) \leq p^{(1)}_{low} + p^{(2)}_{low}, \\
        & p^{(1)}(\vv^*) + p^{(2)}(\vv^*) \leq p^{(1)}(\vv) + p^{(2)}(\vv),\forall \vv \in \mathcal{V}_t.
        \end{aligned}
\end{equation}
We aim to construct optimal solution $(\lambda^*,\mu^*)$ of (\ref{ex_dual_prob_vertex}) and consider two cases:
\begin{itemize}
    \item ($p^{(1)}_{low} + p^{(2)}_{low} > 0$): we
choose $$\lambda^* = 0, (\mu^*)_{\vv} = \begin{cases}
    0 & \text{ if } \vv \not\in \mathcal{V}_t \\
    \frac{-p^{(2)}_{low}}{p^{(1)}_{low} + p^{(2)}_{low}} p^{(1)}(\vv) + \frac{p^{(1)}_{low}}{p^{(1)}_{low} + p^{(2)}_{low}} p^{(2)}(\vv) & \text{ if } \vv \in \mathcal{V}_t
\end{cases}.$$

Note that the above choice of $\mu^*$ is well-defined because the
sub-problems remain feasible in all iterations of the SDA algorithm, that is, $p^{(1)}(v), p^{(2)}(v)$ are finite for all $v \in \mathcal{V}_t$ and $p^{(1)}_{low} + p^{(2)}_{low} > 0$.

We claim that $(\lambda^*,\mu^*)$ is the optimal solution of (\ref{ex_dual_prob_vertex}). To see this, let 
\begin{equation}
    \begin{aligned}
        \label{eq:compare_1}
        L^1(\xv^{(1)}) := \min_{\yv^{(1)}} \ &  \iprod{\cv^{(1)} + \lambda^* }{\xv^{(1)}} + \iprod{\dv^{(1)}}{\yv^{(1)}} + \iprod{\mu^*}{\wv^{(1)}} \\
   \text{s.t.} \ &  (\xv^{(1)},\yv^{(1)},\wv^{(1)}) \in \mathcal{X}^{(1)}_V, \\
    L^2(\xv^{(2)}) := \min_{\yv^{(2)}} \ &  \iprod{\cv^{(2)} - \lambda^*}{\xv^{(2)}} + \iprod{\dv^{(2)}}{\yv^{(2)}} - \iprod{\mu^*}{\wv^{(2)}}  \\
   \text{s.t.} \ &  (\xv^{(2)},\yv^{(2)},\wv^{(2)}) \in \mathcal{X}^{(2)}_V, \\
   u_1 := & \min_{\xv^{(1)} \in \{0,1\}^n} L^1(\xv^{(1)}), \\
   u_2 := & \min_{\xv^{(2)} \in \{0,1\}^n} L^2(\xv^{(2)}).
    \end{aligned}
\end{equation}

We claim that $u_1 = \frac{p^{(1)}_{low}}{p^{(1)}_{low} + p^{(2)}_{low}} (p^{(1)}(\vv^*) + p^{(2)}(\vv^*))$ and $u_2 = \frac{p^{(2)}_{low}}{p^{(1)}_{low} + p^{(2)}_{low}} (p^{(1)}(\vv^*) + p^{(2)}(\vv^*))$. Consider $L^1(\xv^{(1)})$, it is straightforward to see that
\begin{equation}
    \begin{aligned}
        \label{eq:compare_2}
            & L^{1}(\vv^*) = p^{(1)}(\vv^*) + \mu^*_{\vv^*}  =\frac{p^{(1)}_{low}}{p^{(1)}_{low} + p^{(2)}_{low}} (p^{(1)}(\vv^*) + p^{(2)}(\vv^*)), 
    \end{aligned}
\end{equation}
where the first equality follows from the definition of $p^{(1)}(\cdot)$ and the second equality follows by observing that $v^* \in \mathcal{V}_t$ and by plugging in the value of $(\mu^*)_{v^*}$. Therefore, more generally, we also have:
\begin{equation}
    \begin{aligned}
        \label{eq:compare_3}
          &  L^{1}(\vv) =   p^{(1)}(\vv) + \mu^*_{\vv} = \frac{p^{(1)}_{low}}{p^{(1)}_{low} + p^{(2)}_{low}} (p^{(1)}(\vv) + p^{(2)}(\vv)),\forall \vv \in \mathcal{V}_t, \\
    & L^{1}(\vv) = p^{(1)}(\vv) \geq p^{(1)}_{low},\forall \vv \in \{0,1\}^n \setminus \mathcal{V}_t.\\
    \end{aligned}
\end{equation}

By (\ref{two_block_vertex_opt_condition}),(\ref{eq:compare_2}) and (\ref{eq:compare_3}), one can check that
\begin{align*}
    & L^{1}(\vv^*) \leq L^{1}(\vv),\forall \vv \in \mathcal{V}_t, \\
    &  L^{1}(\vv^*) \leq p^{(1)}_{low} \leq  L^{1}(\vv), \forall \vv \in \{0,1\}^n \setminus \mathcal{V}_t.
\end{align*}

Thus one can conclude that $u_1 = L^{1}(\vv^*)$ and similarly we can show that $u_2 = L^{2}(\vv^*)$. 
In this case, we have $L^V(\lambda^*,\mu^*) = u_1 + u_2 = p^{(1)}(\vv^*) + p^{(2)}(\vv^*) = \OPT$. Therefore, $(\lambda^*,\mu^*)$ is the optimal solution of $(\ref{ex_dual_prob_vertex})$. 

Each non-zero entry of $\mu^*$ is bounded by $2\phi$ since $|p^{(i)}(\vv)| \leq \phi$ and $\left|\frac{-p^{(2)}_{low}}{p^{(1)}_{low}}\right|$, $\left|\frac{p^{(1)}_{low}}{p^{(1)}_{low}}\right| \in [0,1]$.
Since $\norm{\mu^*}_0 \leq 2t$, it follows that $\norm{(\lambda^*,\mu^*)}^2 \leq 8 \phi^2 t$.
\item ($p^{(1)}_{low} + p^{(2)}_{low} = 0$) choose $$\lambda^* = 0, \mu^* = 0$$
Since $\cv^{(i)},\dv^{(i)},\xv^{(i)},\yv^{(i)}$ are non-negative, combined with (\ref{two_block_vertex_opt_condition}), $p^{(1)}_{low} + p^{(2)}_{low} = 0$ implies that 
\begin{equation}
    \begin{aligned}
        & p^{(1)}_{low} = p^{(2)}_{low} = 0 \\
        & p^{(1)}(\vv^*) = p^{(2)}(\vv^*) = 0 \\
    \end{aligned}
\end{equation}
We let $L^{(i)}(\xv^{(i)}),u^{(i)}$ be defined in the same way as in (\ref{eq:compare_1}). Since $L^{(i)}(\xv) \geq 0,\forall \xv^{(i)} \in \{0,1\}^n$, then it follows that $u^{(i)} = L^{(i)}(\vv^*) = p^{(i)}(\vv^*) = 0,\forall i \in \{1,2\}$. Therefore $L^V(\lambda^*,\mu^*) = u_1 + u_2 = p^{(1)}(\vv^*) + p^{(2)}(\vv^*)$. By weak duality, $(\lambda^*,\mu^*)$ is the optimal solution of $(\ref{ex_dual_prob_vertex})$ and $\norm{(\lambda^*,\mu^*)} = 0$.
\end{itemize}
\end{proof}

\begin{proposition}
    \label{prop:terminate2}
      If SDA terminates after $t$ iterations because one of the sub-problems is infeasible, then there exists some optimal solution of (\ref{ex_dual_prob_vertex}) $(\lambda^*,\mu^*)$ such that $\norm{(\lambda^*,\mu^*)}^2 \leq 9 \phi^2 t$.
\end{proposition}

\begin{proof}
We let
\begin{align*}
    & \mathcal{Y}^{(i)} := \{ \xv^{(i)} : \exists \yv^{(i)}, (\xv^{(i)},\yv^{(i)}) \in \mathcal{X}^{(i)}\},\forall i \in \{1,2\} \\
    & f^{(i)}(\xv^{(i)}) := \min_{\yv^{(i)}}\left\{ \iprod{\dv^{(i)}}{\yv^{(i)}} : (\xv^{(i)},\yv^{(i)}) \in \mathcal{X}^{(i)}\right\},\forall i \in \{1,2\}.
\end{align*}
Without losing any generality, we assume that 
block $1$ is infeasible after $t$ iterations while block $2$ is feasible in the first $(t-1)$-th iteration. This implies that $|\mathcal{Y}^{(1)}| \leq t$.

Now we choose 
$$\lambda^* = 0, (\mu^*)_{\vv} = \begin{cases}
        0 & \text{ if } \vv \not\in \mathcal{Y}^{(1)} \\
        -\iprod{\cv^{(1)}}{\vv} - f^{(1)}(\vv) + 2 \phi & \text{ if } \vv \in \mathcal{Y}^{(1)}. 
    \end{cases}$$ 
    We claim that $(\lambda^*,\mu^*)$ is the optimal solution of (\ref{ex_dual_prob_vertex}). 
    Let
    \begin{align*}
    L^1(\xv^{(1)}) = \min_{\yv^{(1)}} \ &  \iprod{\cv^{(1)} + \lambda^* }{\xv^{(1)}} + \iprod{\dv^{(1)}}{\yv^{(1)}} + \iprod{\mu^*}{\wv^{(1)}} \\
   \text{s.t.} \ &  (\xv^{(1)},\yv^{(1)},\wv^{(1)}) \in \mathcal{X}^{(1)}_V, \\
    L^2(\xv^{(2)}) = \min_{\yv^{(2)}} \ &  \iprod{\cv^{(2)} - \lambda^*}{\xv^{(2)}} + \iprod{\dv^{(2)}}{\yv^{(2)}} - \iprod{\mu^*}{\wv^{(2)}}  \\
       \text{s.t.} \ &  (\xv^{(2)},\yv^{(2)},\wv^{(2)}) \in \mathcal{X}^{(1)}_V, \\
    \end{align*}
    
    One can verify that
    \begin{align*}
        & L^1(\vv) = \begin{cases}
            \infty & \text{ if } \vv \not\in \mathcal{Y}^{(1)} \\
            2 \phi & \text{ if } \vv \in \mathcal{Y}^{(1)}
        \end{cases}, \\
        & L^2(\vv) = \begin{cases}
            \iprod{\cv^{(2)}}{\vv} + f^{(2)}(\vv) & \text{ if } \vv \not\in \mathcal{Y}^{(1)} \\
            \iprod{\cv^{(1)} + \cv^{(2)}}{\vv} + f^{(1)}(\vv)+ f^{(2)}(\vv) - 2 \phi & \text{ if } \vv \in \mathcal{Y}^{(1)}
        \end{cases}. 
\end{align*}
By the choice of $\mu^*$, it follows that $L^2(\vv) \leq 0 \leq L^2(\uv),\forall \vv \in \mathcal{Y}^{(2)} \cap \mathcal{Y}^{(1)},\uv \in \mathcal{Y}^{(2)} \setminus \mathcal{Y}^{(1)}$. Therefore, it follows that
\begin{align*}
    u_1 & := \min_{\xv^{(1)} \in \{0,1\}^n} L^1(\xv^{(1)}) = 2 \phi, \\
    u_2 & := \min_{\xv^{(2)} \in \{0,1\}^n} L^2(\xv^{(2)}) = \min_{\xv^{(2)} \in \mathcal{Y}^{(1)} \cap \mathcal{Y}^{(2)}} L^2(\xv^{(2)}) \\
    &= \min_{\xv^{(2)} \in \mathcal{Y}^{(1)} \cap \mathcal{Y}^{(2)}}  \iprod{\cv^{(1)} + \cv^{(2)}}{\xv^{(2)}} + f^{(1)}(\xv^{(2)}) + f^{(2)}(\xv^{(2)}) - 2 \phi.
\end{align*}
Therefore, we know that $L^V(\lambda^*,\mu^*) = u_1 + u_2 = \iprod{\cv^{(1)} + \cv^{(2)}}{\xv^{(2)}} + f^{(1)}(\xv^{(2)}) + f^{(2)}(\xv^{(2)}) \text{ for some } \xv^{(2)} \in \mathcal{Y}^{(1)}$. By weak duality, $(\lambda^*,\mu^*)$ is the optimal solution of $(\ref{ex_dual_prob_vertex})$ where each nonzero entry of $\mu$ is bounded by $3 \phi$. Since $\norm{\mu^*}_0 \leq t$, it follows that $\norm{(\lambda^*,\mu^*)}^2 \leq 9 \phi^2 t$. 
\end{proof}
\textbf{(Proof of Proposition \ref{prop_comparision}):}
\begin{proof}
    It follows directly from Proposition $\ref{prop:terminate1}$, Proposition $\ref{prop:terminate2}$ and Proposition \ref{prop:gradient}.
\end{proof}

\subsection{Proof of Proposition \ref{prop:SDA_bad}}
\begin{proof}
    Consider the following example
    \begin{equation}
        \begin{aligned}
            \min & \iprod{-\epsilon\ev}{\xv^{(1)}} + \iprod{\epsilon \ev}{\xv^{(2)}} \\
        \text{s.t. } & \xv^{(1)} = \xv^{(2)}, \\
        & \xv^{(1)} \in \{0,1\}^n, \\
        & \xv^{(2)} \in \{0,1\}^n. \\
        \end{aligned}
    \end{equation}
    One can verify that $(\epsilon \ev,\textbf{0})$ is an optimal solution of (\ref{ex_dual_prob_vertex}) and therefore by Proposition \ref{prop:gradient}, there is a sub-gradient method that solves (\ref{ex_dual_prob_vertex}) to $\epsilon$-optimality in ${n^2(n+2)^2}$ iterations.

    On the other hand, for the SDA algorithm, which fixes $\lambda = 0$, suppose the algorithm already runs $t$ iteration. Let $\mathcal{V}$, $\text{UB}$ and $\text{LB}$ be the same described in the SDA algorithm described above where $\mathcal{V}$ is the collection of points the algorithm removed, $\text{UB}$ is the primal bound and $\text{LB}$ is the dual bound.
It is clear that $|\mathcal{V}| \leq 2t$ and $\text{UB} = 0$. We observe if there exists some pair $\vv_1,\vv_2 \in \{0,1\}^n \setminus \mathcal{V}$ such that $\vv_1 \geq \vv_2$ and $\norm{\vv_1 - \vv_2}_0 \geq 2$, then $\text{LB} \leq -2\epsilon \leq -\epsilon$. To see this:
    \begin{align*}
        \text{LB} & = \min_{\xv^{1},\xv^{(2)} \in \{0,1\}^{n} \setminus \mathcal{V}}  \iprod{-\epsilon\ev}{\xv^{(1)}} + \iprod{\epsilon \ev}{\xv^{(2)}} \\
        & \leq \iprod{-\epsilon\ev}{\vv_1} + \iprod{\epsilon \ev}{\vv_2} \\
        & \leq -2\epsilon.
\end{align*}
Consider any $\vv \in \{0,1\}^{n-2}$, we construct $\vv_1 = (\vv,1,1),\vv_2 = (\vv,0,0)$. Therefore there exists $2^{n -2}$ many distinct pairs of $(\vv_1,\vv_2)$ in $\{0,1\}^n \setminus \mathcal{V}$ such that $\vv_1 \geq \vv_2$ and $\norm{\vv_1 - \vv_2}_0 \geq 2$. Therefore if $|\mathcal{V}| \leq 2^{n-2}$, there must exist such pair and therefore $\text{LB} \leq -2\epsilon$. Therefore, SDA needs at least $2^{n-3}$ iterations to certificate $\epsilon$-optimality.
\end{proof}

\section{Details of preliminary numerical experimental}\label{sec:details}

\subsection{Instance generation}

In the preliminary numerical experimental mentioned in section \ref{sec:compute}, every block is a maximum cardinality stable set problem on a random graph on 75 nodes.
We generated 10 instances for each class of problems. For each instance and each block, we construct a random graph by uniformly picking a density number $d$ from $[0.1,0.15]$ and uniformly sampling with replacement $d \cdot \frac{100\cdot 99}{2}$ edges from a clique of 100 nodes. 

\subsection{Modification of V-Lagrangian}
It is straightforward to see that if $F^{(1)}(\xv^{(1)}) + F^{(2)}(\xv^{(2)}) \geq 0$ is a primal-redundant constraint in (\ref{implied}), then if $\beta (F^{(1)}(\xv^{(1)}) + F^{(2)}(\xv^{(2)})) \geq 0$ is also a  a primal-redundant constraint for any $\beta \geq 0$. In the numerical experiment, the primal-redundant constraint in VL takes the form
    \begin{equation}
        \begin{aligned} \label{eq:V-cuts11}
           2 \cdot \prod_{j \in [n]} \sigma_{v_j}(x_j^{(1)}) = 2 \cdot \prod_{j \in [n]} \sigma_{v_j}(x_j^{(2)}),\text{ for each vertex } \vv \text{ of } [0,1]^{n}.
        \end{aligned}
    \end{equation}
We empirically observe that such modification increases numerical performance.

\subsection{Choice of bundle methods for various Lagrangian Dual}

In this subsection, we give some implementation details of the bundle method \cite{bagirov2014introduction} we use to update multipliers in Lagrangian decomposition. Among many variants of bundle methods, we use the proximal bundle method (without changing analytic center) with adaptive step size and the bundle-level method (without changing analytic center). We empirically observed that both the bundle methods have similar performance when applying to L. On the other hand, we also observed that the proximal bundle method has a better numerical performance when applied to QL, while the bundle-level method has a better numerical performance when applied to VL. Therefore, in the numerical experimental, we used the proximal bundle method for both L and QL and we used the bundle-level method for VL. For the proximal bundle method, we use adaptive step size motivated by the famous Polyak step size \cite{polyak2021introduction} and recently numerically verified by \cite{monteiro2024parameter}. 
\begin{algorithm}
\label{alg:prox_bundle}
  \caption{A proximal bundle method}
  \Input{$f(\cdot),\text{LB}$}
  
  \# $f$ is a convex function and $\text{LB}$ is a lower bound of $f(\lambda)$
  
  Initialize $\lambda_1 = 0$
  $\text{UB} \leftarrow \infty,i=1$ 
  
    \While{\text{some termination criteria is not fulfilled}}{
    compute $f(\lambda_i),\nabla f(\lambda_i)$

    update LB be the optimal value of the following problem
       \begin{align*}
        \min_{t,\lambda} \ & t \\
        \text{s.t. } & t \geq \text{LB} \\
        & t \geq f(\lambda_j) + \iprod{\nabla f(\lambda_j)}{\lambda - \lambda_j},\forall j \in [i] 
    \end{align*}

    choose $\alpha = \frac{f(\lambda_i) - LB}{\norm{\nabla f(\lambda_i)}^2}$

    Let $\lambda_{i+1}$ be the optimal solution of the following problem
    \begin{align*}
        \min_{t,\lambda} & \norm{\lambda - \lambda_{i}}^2 + \alpha t \\
        \text{s.t. } & t \geq \text{LB} \\
        & t \geq f(\lambda_j) + \iprod{\nabla f(\lambda_j)}{\lambda - \lambda_j},\forall j \in [i] 
    \end{align*}
    $\text{UB} = \min\{\text{UB},f(\lambda_i)\},i = i + 1$
    }
\end{algorithm}

\begin{algorithm}
\label{alg:level_bundle}
  \caption{A bundle-level method}
  \Input{$f(\cdot),\text{LB}$}
  
  \# $f$ is a convex function and $\text{LB}$ is a lower bound of $f(\lambda)$
  
  Initialize $\lambda_1 = 0$
  $\text{UB} \leftarrow \infty,i=1$ 
  
    \While{\text{some termination criteria is not fulfilled}}{
    compute $f(\lambda_i),\nabla f(\lambda_i)$

    update LB be the optimal value of the following problem
       \begin{align*}
        \min_{t,\lambda} \ & t \\
        \text{s.t. } & t \geq \text{LB} \\
        & t \geq f(\lambda_j) + \iprod{\nabla f(\lambda_j)}{\lambda - \lambda_j},\forall j \in [i] 
    \end{align*}

    choose $\alpha = 0.3$

    Let $\lambda_{i+1}$ be the optimal solution of the following problem
    \begin{align*}
        \min_{t,\lambda} & \norm{\lambda - \lambda_{i}}^2 \\
        \text{s.t. } & t \leq \alpha \cdot \text{LB} + (1-\alpha) \cdot \text{UB} \\
        & t \geq f(\lambda_j) + \iprod{\nabla f(\lambda_j)}{\lambda - \lambda_j},\forall j \in [i] 
    \end{align*}
    $\text{UB} = \min\{\text{UB},f(\lambda_i)\},i = i + 1$
    }
\end{algorithm}

\end{document}